\documentclass[11pt]{article}

\usepackage{tikz}
\usepackage{adjustbox}
\usepackage{subcaption}
\usepackage{hyperref}
\usepackage[normalem]{ulem}
\usepackage{bm}
\usepackage{amsmath}
\usepackage{amssymb}
\usepackage{amsthm}
\usepackage{mathabx}
\usepackage{amscd}
\usepackage{amsfonts}
\usepackage{graphicx}
\usepackage{fancyhdr}
\usepackage{color}
\usepackage{enumerate}
\usepackage{psfrag} 
\usepackage{epsfig}
\usepackage{epstopdf}
\usepackage{amsbsy}
\usepackage{latexsym}
\usepackage{moreverb}                      
\usepackage{textcomp}
\usepackage{dsfont}
\usepackage{algorithm}
\usepackage{multicol}
\usepackage{algorithmic}
\usepackage{enumitem}

\usepackage[inkscapelatex=false]{svg}
\usepackage{float}

\newtheorem{proposition}{Proposition}
\newtheorem{corollary}{Corollary}
\newtheorem{remark}{Remark}

\newtheorem{theorem}{Theorem}

\usepackage[top=2.8cm,bottom=2.8cm,left=2.5cm,right=2.5cm]{geometry}
\usepackage[T1]{fontenc}
\usepackage[utf8]{inputenc}

\usepackage[affil-it]{authblk} 
\usepackage{etoolbox}
\usepackage{lmodern}

\makeatletter
\patchcmd{\@maketitle}{\LARGE \@title}{\fontsize{18}{19.2}\selectfont\@title}{}{}
\makeatother

\usepackage{sectsty}
\sectionfont{\fontsize{13}{15}\selectfont}
\subsectionfont{\fontsize{12}{15}\selectfont}

\begin{document}

\title{Symmetrized Sinkhorn-Gibbs Inference\\ for Oscillatory Inverse Problems}
\author[1,2]{Gabriel Huerta\thanks{jghuert@sandia.gov, ghuerta@unm.edu}}
\author[1]{Mohammad Motamed\thanks{motamed@unm.edu}}
\affil[1]{Department of Mathematics and Statistics, University of New Mexico, Albuquerque, NM 87131, USA}
\affil[2]{Sandia National Laboratories, Albuquerque, NM 87123, USA}

\date{}

\maketitle

\hrule

\medskip

\noindent
{\bf Abstract.} 
Oscillatory inverse problems often exhibit highly nonconvex discrepancy landscapes due to signal misalignment and cycle-skipping phenomena, posing significant challenges for uncertainty quantification. While Gibbs posteriors provide a flexible framework for incorporating problem-specific discrepancy measures, 
their performance depends strongly on the empirical risk landscape induced by the chosen discrepancy measure. 
Optimal transport accounts for the spatial and temporal structure of the underlying feature domain when comparing oscillatory signals. However, the direct application of classical optimal transport is precluded because observed and predicted signals are typically signed and therefore do not satisfy the positivity requirements of classical transport formulations. We introduce a symmetrized Sinkhorn-Gibbs inference framework for oscillatory inverse problems. The proposed approach combines a normalization procedure for signed signals with a symmetrized Sinkhorn loss that exploits complementary transport information from the normalized signals and their normalized negations. The resulting loss is incorporated into the Gibbs posterior framework, yielding a Gibbs inference methodology tailored to oscillatory data. We establish smoothness properties of the proposed loss, prove well-definedness of the resulting Gibbs posterior, derive robustness guarantees, and develop an adaptive sampling strategy for posterior computation. Numerical experiments demonstrate less oscillatory empirical risk landscapes with fewer spurious local minima, more accurate posterior inference, greater robustness to observational noise, and improved population-level recovery than Gibbs posteriors based on Euclidean and trace-wise Wasserstein losses.

\medskip
\noindent
{\bf Keywords.} Inverse problems,
Gibbs posterior,
Optimal transport,
Sinkhorn divergence,
Uncertainty quantification,
Oscillatory signals,
Cycle skipping,
Bayesian inference.

\bigskip
\hrule

\medskip
\medskip


\section{Introduction}
\label{sec:intro}

Inverse problems arise across science and engineering, where the objective is to infer unknown model parameters from indirect and often noisy observations generated by a physical system. Such problems appear in applications including seismic imaging, wave propagation, medical imaging, and quantum characterization, and are frequently formulated through parameterized forward models governed by systems of ordinary or partial differential equations. In many settings, uncertainty quantification is required in addition to point estimation, motivating Bayesian approaches that characterize probability distributions over the unknown parameters rather than just producing a single estimate. Bayesian inverse problems have been extensively studied in both finite- and infinite-dimensional settings; see, e.g., \cite{Stuart:10,Dashti_Stuart:16,Sullivan:2017}.

A persistent challenge in many inverse problems is the presence of highly oscillatory data. In wave-based applications, small perturbations of the model parameters may induce substantial shifts in the locations of peaks, troughs, and wavefronts. As a consequence, discrepancy measures based on pointwise comparisons often produce highly nonconvex objective functions containing numerous local minima. This phenomenon, arising from temporal and spatial misalignment between observed and predicted signals and commonly referred to as \emph{cycle skipping} in the seismic inversion literature, has long been recognized as a major obstacle in seismic inversion and related wave-propagation problems; see, e.g., \cite{Virieux_Operto:09}. The resulting nonconvexity affects not only optimization algorithms but also uncertainty quantification procedures, since the empirical risk landscape induced by the discrepancy measure strongly influences the resulting posterior distribution.

Gibbs posteriors provide a flexible alternative to traditional Bayesian inference by replacing the likelihood function in Bayes' theorem with a user-specified empirical risk \cite{Bissiri_etal:2016}. Rather than requiring a complete statistical model for the observational noise, Gibbsian inference combines prior information with a loss function that measures the discrepancy between observed and predicted data. This construction has attracted increasing attention in settings where likelihood models are difficult to specify or potentially misspecified, while still retaining many of the computational and conceptual advantages of Bayesian inference \cite{SyringMartin2019,ZafarNicholls2024,LeeLiuNicholls2025}. 
The performance of a Gibbs posterior depends strongly on the empirical risk landscape induced by the chosen loss, including its nonconvexity, local minima, and sensitivity to signal misalignment, making the design of robust and informative discrepancy measures a central problem.

Optimal transport has emerged as a powerful framework for comparing probability distributions by accounting for distances on the underlying feature domain. In inverse problems involving oscillatory signals, transport-based discrepancies often exhibit substantially more favorable optimization landscapes than traditional Euclidean losses \cite{Engquist_Froese:14,Yang_etal:18,Engquist_Yang:19,Yang_Engquist:18,Engquist_Yang:19b,Motamed_Appelo:19}. In particular, optimal transport can mitigate cycle-skipping effects by accounting for the spatial and temporal locations at which signal values occur, rather than comparing amplitudes pointwise. However, classical Wasserstein distances become computationally expensive when the underlying feature domain is multidimensional, as the cost of solving the associated optimal transport problem grows rapidly with dimension. Consequently, previous transport-based approaches to oscillatory inverse problems have typically applied Wasserstein distances only along the temporal dimension and then aggregated the resulting discrepancies across spatial locations using Euclidean norms \cite{Motamed_Appelo:19,DunlopYang2022}. While computationally efficient, this hybrid construction reintroduces the limitations of Euclidean comparisons in the remaining feature dimensions. Entropic regularization addresses this limitation through the Sinkhorn divergence, making multidimensional transport practical while preserving transport information across all feature dimensions \cite{Cuturi:13,Solomon_etal:15,Motamed_Sinkhorn:20}. Nevertheless, the application of transport-based discrepancies to oscillatory inverse problems remains challenging because observed and predicted signals generally do not satisfy the assumptions required by classical optimal transport.

Classical optimal transport compares probability measures and therefore requires inputs to be nonnegative and normalized. Observed and predicted signals in scientific applications, however, are often signed, oscillatory, and not naturally mass preserving. This issue has motivated a variety of normalization strategies that transform oscillatory signals into probability measures prior to the application of transport-based distances \cite{Engquist_Froese:14,Yang_etal:18,FrenchOT3,Engquist_Yang:19,Yang_Engquist:18,Engquist_Yang:19b,Motamed_Appelo:19}. While these approaches enable the use of optimal transport, a single nonlinear normalization may suppress information associated with one phase of a signed oscillatory signal. Moreover, normalization can alter the favorable convexity properties of transport-based discrepancies, leading to additional nonconvexity in the resulting empirical risk landscape. These limitations motivate a construction that combines complementary normalized representations while recovering favorable asymptotic convexity behavior.

The objective of this work is to develop a transport-based Gibbsian inference framework specifically tailored to oscillatory inverse problems. We introduce a symmetrized Sinkhorn-Gibbs posterior that combines a normalization procedure for signed signals with a symmetrized Sinkhorn loss constructed from the normalized signals and their normalized negations. Because normalization is nonlinear, the normalized representations of a signal and its negation contain complementary information about its positive and negative phases. By combining these representations, the proposed loss reduces the information loss associated with a single normalization and recovers favorable convexity behavior in the asymptotic regime. 
Together with multidimensional Sinkhorn transport, this construction yields empirical risk landscapes with fewer spurious oscillations and competing local minima, leading to more accurate posterior inference. The resulting framework remains computationally tractable through entropic regularization and can be integrated directly into standard Gibbs posterior methodologies.

The main contributions of this work are summarized as follows.

\begin{itemize}

\item {\bf Contribution 1:}
We introduce a symmetrized Sinkhorn-Gibbs inference framework for inverse problems with oscillatory signals. The framework is independent of the particular forward model and can be coupled with physics-based simulators, reduced-order models, or data-driven surrogates without modification. By combining Gibbsian inference with entropic optimal transport, it performs transport comparisons across the full feature domain, rather than only along individual traces as in previous transport-based approaches.

\item {\bf Contribution 2:}
We develop a normalization and symmetrization strategy that enables transport-based inference for signed oscillatory data. The construction transforms oscillatory signals into probability measures while combining complementary normalized representations to preserve information from both phases of the oscillation.

\item {\bf Contribution 3:}
We establish theoretical properties of the proposed framework, including smoothness of the normalized Sinkhorn divergence, asymptotic convexity and well-definedness of the resulting Gibbs posterior, and robustness of the posterior distribution with respect to perturbations in both the observed data and the forward model.

\item {\bf Contribution 4:}
Through a series of challenging seismic inversion benchmarks, we demonstrate that the proposed framework yields empirical risk landscapes with fewer spurious oscillations and competing local minima, more accurate posterior inference, greater robustness to observational noise, and improved population-level recovery than Gibbs posteriors based on Euclidean and trace-wise Wasserstein losses.

\end{itemize}

The remainder of the paper is organized as follows. Section~\ref{sec:background} introduces the inverse problem setting, reviews Bayesian and Gibbsian inference, and develops the transport-based framework underlying the proposed methodology. Section~\ref{sec:inf_OT} presents the construction of the symmetrized Sinkhorn-Gibbs posterior together with an adaptive sampling strategy for jointly inferring model parameters and the inverse temperature, a learning rate in the Gibbs posterior framework. Section~\ref{sec:analysis} establishes theoretical properties of the proposed framework, including smoothness, well-definedness, convexity-related behavior, and posterior robustness. Section~\ref{sec:numerical_exp} presents numerical experiments illustrating the performance of the method on challenging oscillatory inverse problems. 
Finally, Section~\ref{sec:conclusions} summarizes the main findings and discusses directions for future research.

\section{Problem Statement and Background}
\label{sec:background}

In this section, we formulate the inverse problem of recovering model parameters from noisy measurements supported on a compact Euclidean domain, where observations arise from a parametric forward model, such as systems of ordinary or partial differential equations (ODEs/PDEs) depending on \(n_p\) unknowns. We present both Bayesian and Gibbsian inference frameworks for this problem, highlighting their respective modeling assumptions and motivations. Particular attention is given to the role of the loss function in Gibbsian inference and the use of transport-based dissimilarity measures for comparing oscillatory data. We also establish the connection between continuous and discrete representations of the inverse problem, providing the theoretical and computational foundations for the symmetrized Sinkhorn-Gibbs framework developed in the following sections.

\subsection{Inverse Problem Setup}
\label{sec:inv_setup}

Let $X \subset \mathbb{R}^{n_d}$ be a compact \emph{feature} domain, equipped with the 
Lebesgue measure, where $n_d \in \mathbb{N}$ denotes the ambient dimension. In many 
applications, including those considered in this paper, $X$ represents a physical 
spatial and/or temporal domain in one to four dimensions ($n_d=1,2,3,4$), such as a 
purely temporal interval, a static spatial region, a spatial region evolving over time, 
or a multidimensional grid arising in experimental measurements.

We define the function space $\mathcal{F}(X)$ as the set of bounded, real-valued 
functions $h$ on $X$ with zero mean:
\begin{equation}\label{FX1}
\mathcal{F}(X) := \left\{ h: X \rightarrow {\mathbb R};\ h \in L^{\infty}(X) 
\,\middle|\, \int_X h(\mathbf{x})\, d\mathbf{x} = 0 \right\}.
\end{equation}

Consider a parametric operator
\[
f : \Theta \longrightarrow \mathcal{F}(X),
\]
representing a \emph{forward model} that maps model parameters $\boldsymbol{\theta} \in \Theta \subset 
\mathbb{R}^{n_p}$, with $\Theta$ a compact parameter space, to the function space \eqref{FX1}. The forward model may arise from an ODE or PDE system, a computational surrogate, a reduced-order model, or any other mechanism that provides evaluations of the parameter-to-signal map. For each $\boldsymbol{\theta} \in \Theta$, the forward model produces a 
function $f(\boldsymbol{\theta}) \in \mathcal{F}(X)$, with pointwise evaluations 
$f(\boldsymbol{\theta})(\mathbf{x})$ for every $\mathbf{x} \in X$. We retain the 
operator notation $f(\boldsymbol{\theta})(\mathbf{x})$ to emphasize that $f$ first maps 
parameters to a function on $X$, which is then evaluated at $\mathbf{x}$.

We distinguish the forward model from the experimentally measured data. We introduce a 
measured data field
\[
    g : X \to \mathbb{R}, \qquad g \in \mathcal{F}(X),
\]
which represents the observed data obtained from the experiment. In practice, $g$ is only 
observed at finitely many locations $\{\mathbf{x}_i\}_{i=1}^n \in X$, yielding the 
discrete observations $\{ g(\mathbf{x}_i) \}_{i=1}^n$.

Given the forward map $f$ and these observations, the inverse problem consists of inferring information about the model parameters $\boldsymbol{\theta}$ such that 
the forward-model output $f(\boldsymbol{\theta})$ matches the measured field $g$ at the 
sampled points, in an appropriate sense (exactly in the noise-free case, or 
approximately in the presence of measurement error).

Typical datasets of interest include: (i) wavefield or seismic data supported on spatiotemporal grids in two to four dimensions, obtained from seismograms; and (ii) one- or two-dimensional measurements of state populations in quantum devices, acquired through experiments such as Rabi or Ramsey protocols. In all cases, the data can be viewed as $n$ measured values $\{ g(\mathbf{x}_i) \}_{i=1}^n$ sampled at $n$ grid points $\{\mathbf{x}_i\}_{i=1}^n$ in the compact domain $X$.

Many inverse problems arising from such datasets involve highly oscillatory signals, which often give rise to challenging optimization and inference landscapes. As a result, the choice of discrepancy measure used to compare observed data with model predictions can have a significant impact on the quality of the resulting parameter estimates and uncertainty quantification.

\begin{remark}
In an idealized, noise-free setting, one may assume the existence of a 
$\theta^\ast \in \Theta$ such that $g = f(\theta^\ast)$. A common noisy variant is the 
additive-noise model $g(\mathbf{x}_i) = f(\theta^\ast)(\mathbf{x}_i) + \varepsilon_i$. 
Our general formulation does not require either assumption.
\end{remark}

\subsection{Continuous vs. Discrete Representations}

Our analysis and algorithms involve two distinct yet complementary perspectives on the 
data: the continuous and discrete settings.

In the \emph{continuous} setting, measured and simulated data are modeled as functions 
or probability measures over the compact domain $X$. These induce continuous measures 
$\mu_f$ and $\mu_g$, which serve as the primary objects of study in optimal transport. 
This formulation is crucial for establishing rigorous stability and well-posedness 
properties, as developed in Section~\ref{sec:analysis}.

In contrast, in practical computations we only have access to data at finitely many 
locations $\{\mathbf{x}_i\}_{i=1}^n \subset X$. This yields empirical measures $F_n$ and 
$G_n$ supported on point clouds or a discrete set of points, represented by probability vectors $\mathbf{f}_n$ and 
$\mathbf{g}_n$ with positive entries summing to one. Algorithms for optimal transport, 
including Sinkhorn’s algorithm, operate in this discrete setting. We describe these 
computational aspects, along with MCMC methods for sampling from the Sinkhorn–Gibbs 
posterior, in detail in Section~\ref{sec:inf_OT}.

By explicitly connecting the continuous and discrete representations, we bridge the gap 
between theoretical guarantees and computational feasibility. This dual perspective 
underpins our entire framework, allowing us to develop methods that are both robust in 
theory and efficient in practice.

\subsection{Bayesian and Gibbsian Inference}

\paragraph{Bayesian inference.}
The classical Bayesian approach to inverse problems \cite{Kaipo_Somersalo:05,Stuart:10} models the unknown parameters $\boldsymbol{\theta}$ as random variables with a prior density $\pi(\boldsymbol{\theta})$. Given a statistical model for the data, specified through a likelihood $\pi(\mathbf{g}_n | \boldsymbol{\theta})$, Bayes' theorem yields the posterior density 
$$
\pi(\boldsymbol{\theta} | \mathbf{g}_n) = \frac{\pi(\mathbf{g}_n | \boldsymbol{\theta}) \, \pi(\boldsymbol{\theta}) } {\int_\Theta \pi(\mathbf{g}_n | \boldsymbol{\theta}) \, \pi(\boldsymbol{\theta}) \, d\boldsymbol{\theta}},
$$
where $\mathbf{g}_n = \{ g(\mathbf{x}_i) \}_{i=1}^n$ denotes the measurements. Once a prior and likelihood are specified, well-established computational methods exist for estimating the posterior, including Markov Chain Monte Carlo (MCMC) and variational inference techniques \cite{VI:17}.

In practice, however, specifying an accurate likelihood is often challenging or infeasible, which can lead to model misspecification and inaccurate posterior inference.

\paragraph{Gibbsian inference.}
The Gibbsian (or general Bayesian) framework \cite{Bissiri_etal:2016,Zhang:2006,Catoni:2007} offers a flexible alternative that constructs posteriors by updating priors using a loss function instead of a likelihood. Specifically, one selects a loss $L$ measuring the dissimilarity between observed data $\mathbf{g}_n = \{ g(\mathbf{x}_i) \}_{i=1}^n$ and simulated data $\mathbf{f}_n(\boldsymbol{\theta}) = \{ f(\boldsymbol{\theta})(\mathbf{x}_i) \}_{i=1}^n$. The empirical risk is then defined as
\begin{equation} \label{risk_loss}
\Phi(\boldsymbol{\theta}; \mathbf{g}_n) := L\left( \mathbf{f}_n(\boldsymbol{\theta}), \mathbf{g}_n\right),
\end{equation}
and the discrete Gibbs posterior is
\begin{equation} \label{Gibbs_posterior_discrete}
\Pi_n(d\boldsymbol{\theta}) = \frac{\exp\left(- \lambda n \, \Phi(\boldsymbol{\theta}; \mathbf{g}_n)\right) \, \Pi_0(d\boldsymbol{\theta})}{\int_\Theta \exp\left(- \lambda n \, \Phi(\boldsymbol{\theta}; \mathbf{g}_n)\right) \, \Pi_0(d\boldsymbol{\theta})},
\end{equation}
where $\Pi_0(d\boldsymbol{\theta})$ is the prior measure, the factor $n$ is the number of observations, and $\lambda>0$ is the inverse temperature, also referred to as the learning rate in the Gibbs posterior literature.

This formulation admits a decision-theoretic interpretation through the variational framework of \cite{Bissiri_etal:2016}. Assuming that the prior measure $\Pi_0$ admits a density $\pi_0$ with respect to a suitable reference measure, the Gibbs posterior $\Pi_n$ likewise admits a density $\pi_n$. The posterior density may then be characterized as the unique minimizer of the cost functional
\[
C(\pi_n) := \lambda n \int \Phi(\boldsymbol{\theta}, \mathbf{g}_n) \, \pi_n(\boldsymbol{\theta}) \, d\boldsymbol{\theta} + \int \pi_n(\boldsymbol{\theta}) \, \log \frac{\pi_n(\boldsymbol{\theta})}{\pi_0(\boldsymbol{\theta})} \, d\boldsymbol{\theta},
\]
which balances fidelity to data through the expected empirical risk and coherence with prior beliefs through the Kullback-Leibler divergence. The unique minimizer satisfies
\[
\pi_n(\boldsymbol{\theta}) \propto \exp\left(- \lambda n \, \Phi(\boldsymbol{\theta}, \mathbf{g}_n)\right) \pi_0(\boldsymbol{\theta}),
\]
which is precisely the density corresponding to the Gibbs posterior measure defined above.


Although our discrete formulation \eqref{Gibbs_posterior_discrete} serves as the basis for computational algorithms, it naturally approximates a continuous setting where the data are modeled as functions $g \in \mathcal{F}(X)$ and $f(\boldsymbol{\theta}) \in \mathcal{F}(X)$ over $X$. In this continuous perspective, the empirical risk $\Phi(\boldsymbol{\theta}; \mathbf{g}_n)$ is replaced by a risk functional $\Phi(\boldsymbol{\theta}; g)$ measuring dissimilarity between the continuous functions. The resulting continuous Gibbs posterior is given by
\begin{equation} \label{Gibbs_posterior_continuous}
\Pi(d\boldsymbol{\theta}) = \frac{\exp\left(- \lambda \, \Phi(\boldsymbol{\theta}; g)\right) \, \Pi_0(d\boldsymbol{\theta})}{\int_\Theta \exp\left(- \lambda \, \Phi(\boldsymbol{\theta}; g)\right) \, \Pi_0(d\boldsymbol{\theta})},
\end{equation}
which provides the foundation for the stability and well-posedness analysis presented in Section~\ref{sec:analysis}.

The factor $n$ in \eqref{Gibbs_posterior_discrete} accounts for the normalization implicit in empirical risks, which are typically defined as averages over the observations. For instance, when the empirical risk is induced by a sample-wise loss $\ell$,
$$
\Phi(\boldsymbol{\theta}; \mathbf{g}_n) = \frac{1}{n} \sum_{i=1}^n \ell(f(\boldsymbol{\theta})(\mathbf{x}_i), g(\mathbf{x}_i)),
$$
then $n \, \Phi(\boldsymbol{\theta}; \mathbf{g}_n)$ represents the cumulative contribution of all observations. 
In the continuous formulation, this scaling is absorbed into the definition of the risk functional $\Phi(\boldsymbol{\theta}; g)$, which measures the total discrepancy between continuous fields.

While Gibbs posteriors are not derived directly from Bayes' theorem, they nonetheless retain a Bayesian form and can be computed using standard Bayesian algorithms without requiring a fully specified likelihood function.

The performance of a Gibbs posterior depends critically on the choice of loss function through the empirical risk landscape it induces over the parameter space. In oscillatory inverse problems, conventional pointwise discrepancies may assign large penalties to signals that are visually similar but slightly misaligned, resulting in highly nonconvex risk landscapes with numerous local minima. This phenomenon, commonly referred to as \emph{cycle skipping} in waveform inversion, motivates the development of transport-based loss functions that better capture geometric similarities between observed and simulated data.

\subsection{\bf Optimal Transport Dissimilarity Measures} 

Optimal transport provides a powerful feature-domain-aware framework for comparing probability measures; see e.g., \cite{Villani:03,Villani:09,Ambrosio_Gigli:13,Santambrogio:15}. In this section, we first review optimal transport in the continuous setting, which forms the mathematical foundation for our analysis, and then describe its discrete counterpart, which is essential for practical computations on finite datasets. By connecting these two formulations, we link theoretical properties of optimal transport divergences with their numerical approximations used in our algorithmic framework.

\subsubsection{Continuous optimal transport}
\label{sec:OT_continuous}

Given two probability measures $\mu_f$ and $\mu_g$ on a compact domain $X \subset \mathbb{R}^{n_d}$ and a cost function $c: X \times X \to [0, \infty)$, the Kantorovich formulation defines the optimal transport cost as
\[
\mathrm{OT}(\mu_f, \mu_g) := \inf_{\pi \in \Pi(\mu_f, \mu_g)} \int_{X \times X} c(\mathbf{x}, \mathbf{y}) \, d\pi(\mathbf{x}, \mathbf{y}),
\]
where $\Pi(\mu_f, \mu_g)$ denotes the set of joint probability measures (or couplings) on $X \times X$ with marginals $\mu_f$ and $\mu_g$. For the quadratic cost $c(\mathbf{x}, \mathbf{y})=\|\mathbf{x}-\mathbf{y}\|_2^2$, the celebrated quadratic Wasserstein distance is
\[
W_2(\mu_f, \mu_g) = \left( \mathrm{OT}(\mu_f, \mu_g) \right)^{1/2}.
\]
This defines a true metric on the space of probability measures with finite second moments, satisfying symmetry, positive definiteness, and the triangle inequality.

While $W_2$ has appealing geometric and statistical properties \cite{Villani:09,Santambrogio:15,Peyre_Cuturi:19,Panareto_Zemel:19}, such as convexity and noise-insensitivity, computing it directly is often intractable in high dimensions. To address this, Cuturi \cite{Cuturi:13} proposed an entropic regularization of the OT cost,
\[
\mathrm{OT}_\varepsilon(\mu_f, \mu_g) := \inf_{\pi \in \Pi(\mu_f, \mu_g)} \int_{X \times X} c(\mathbf{x}, \mathbf{y}) \, d\pi(\mathbf{x}, \mathbf{y}) - \varepsilon H(\pi),
\]
where $H(\pi) = - \int_{X \times X} \log \left(\frac{d\pi}{d\mu_f \otimes d\mu_g}\right) \, d\pi$ denotes the entropy of $\pi$ relative to the product measure $\mu_f \otimes \mu_g$, and $\varepsilon > 0$ is a small regularization parameter. 
Built on this regularized transport cost, we define the \emph{Sinkhorn divergence} as
\begin{equation}\label{eq:Sinkhorn_continuous}
S_\varepsilon(\mu_f, \mu_g) := \left(\mathrm{OT}_\varepsilon(\mu_f, \mu_g) - \tfrac{1}{2} \mathrm{OT}_\varepsilon(\mu_f, \mu_f) - \tfrac{1}{2} \mathrm{OT}_\varepsilon(\mu_g, \mu_g) \right)^{1/2}.
\end{equation}
This divergence is symmetric, positive definite (i.e., it vanishes if and only if $\mu_f=\mu_g$), and scales consistently with the classical Wasserstein distance, satisfying
\[
\lim_{\varepsilon \to 0} S_\varepsilon(\mu_f, \mu_g) = W_2(\mu_f, \mu_g).
\]
However, $S_\varepsilon$ does not in general satisfy the triangle inequality and is therefore not a metric. Even weaker, local versions of the triangle inequality may fail to hold unless the measures are suitably close to one another.

\begin{remark}
In the optimal transport literature, including \cite{Feydy_etal:19}, the debiased Sinkhorn divergence is typically defined in squared form, corresponding to our $S_\varepsilon^2$. Here, we adopt the rooted form $S_\varepsilon$ to match the scale of $W_2$ directly.
\end{remark}

In practice, the probability measures $\mu_f$ and $\mu_g$ are represented via finite samples, leading naturally to empirical discrete measures and discrete OT formulations, which we detail next.

\subsubsection{Discrete optimal transport}

Consider two discrete probability measures on $X$:
\[
F_n = \sum_{i=1}^n f_i \, \delta_{\mathbf{x}_i}, \qquad G_n = \sum_{i=1}^n g_i \, \delta_{\mathbf{x}_i},
\]
where $\delta_{\mathbf{x}_i}$ is the Dirac point-mass measure at point $\mathbf{x}_i \in X$, and the weights $\{ (f_i, g_i) \}_{i=1}^n$ form two probability vectors ${\bf f}_n := (f_1, \dotsc, f_n) \in \mathbb{R}_+^n$ and ${\bf g}_n := (g_1, \dotsc, g_n) \in \mathbb{R}_+^n$, summing to 1. For notational simplicity, we assume both measures share the same support.

Let $P = [P_{ij}] \in \mathbb{R}_+^{n \times n}$ denote a transport matrix, i.e., a non-negative matrix with row and column sums equal to ${\bf f}_n$ and ${\bf g}_n$, respectively. Each element $P_{ij}$ represents the amount of mass transported from source point $\mathbf{x}_i$ to target point $\mathbf{x}_j$. For the quadratic cost matrix $C = [C_{ij}]  \in \mathbb{R}_+^{n \times n}$, with components $C_{ij}=||\mathbf{x}_i-\mathbf{x}_j||_2^2$, the quadratic Wasserstein distance is defined as
\begin{equation}\label{W2}
W({\bf f}_n, {\bf g}_n) = \bigl( \min_{P \in U({\bf f}_n, {\bf g}_n)} \langle P, C \rangle \bigr)^{1/2}, 
\quad 
U({\bf f}_n, {\bf g}_n) := \left\{ P \in \mathbb{R}_+^{n \times n} : P \mathds{1}_n = {\bf f}_n, \; P^{\top} \mathds{1}_n = {\bf g}_n \right\},
\end{equation}
where $\langle P, C \rangle = \sum_{i,j} P_{ij} C_{ij}$ is the total transport cost associated with $P$, and $\mathds{1}_n$ is the vector of $n$ ones. 
Direct computation of the Wasserstein distance \eqref{W2} is costly, requiring $\mathcal{O}(n^3)$ time in general \cite{Peyre_Cuturi:19}.

Cuturi’s entropic regularization \cite{Cuturi:13} and subsequent developments \cite{Ramdas_etal:17,Genevay_etal:18,Salimans_etal:2018,Sanjabi_etal:18,Feydy_etal:19,Chizat_etal:20} address the computational bottleneck of OT by solving the entropy-regularized problem
\begin{equation}\label{OT_Regularized}
T_{\varepsilon}({\bf f}_n, {\bf g}_n) := \min_{P \in U({\bf f}_n, {\bf g}_n)} \langle P, C \rangle - \varepsilon H(P), \qquad H(P) = - \sum_{i,j} P_{ij} (\log P_{ij} - 1),
\end{equation}
leading to the discrete Sinkhorn divergence
\begin{equation}\label{S_unbiased}
S({\bf f}_n, {\bf g}_n) = \left( T_{\varepsilon}({\bf f}_n, {\bf g}_n) - \tfrac{1}{2} \left[ T_{\varepsilon}({\bf f}_n, {\bf f}_n) + T_{\varepsilon}({\bf g}_n, {\bf g}_n) \right] \right)^{1/2}.
\end{equation}
The Sinkhorn divergence \eqref{S_unbiased} inherits desirable properties including convexity, smoothness, and symmetric positive definiteness \cite{Feydy_etal:19}. 
Moreover, efficient implementations of Sinkhorn's algorithm and related acceleration techniques make large-scale transport computations computationally feasible; see, e.g., \cite{Sinkhorn:64,Sinkhorn_Knopp:67,Solomon_etal:15,Altschuler_etal:17,Dvurechensky_etal:18,Chizat_etal:20,Motamed_Sinkhorn:20} for details.

\subsection{Transport-Based Losses for Signed Oscillatory Data}
\label{sec:OT_loss}

Among the various transport-based discrepancies available for comparing probability measures, the Sinkhorn divergence offers a particularly attractive combination of displacement-aware comparison, smoothness, and computational efficiency. A key challenge in using the Sinkhorn divergence as the loss function within the Gibbs posterior framework, however, is that classical optimal transport compares probability measures and therefore requires inputs to be nonnegative and normalized to have unit total mass. In many scientific inverse problems, observed and simulated signals are naturally signed and may not satisfy this requirement. Although unbalanced optimal transport formulations \cite{UOT:18,UOT:19,Sejourne_Etal:2021} relax the equal-mass constraint, they still operate on nonnegative measures and therefore do not directly address the challenges posed by signed data.

To address this difficulty, we build upon normalization strategies that transform oscillatory signals into probability measures prior to the application of transport-based divergences \cite{Engquist_Froese:14,Yang_etal:18,FrenchOT3,Engquist_Yang:19,Yang_Engquist:18,Engquist_Yang:19b,Motamed_Appelo:19}. While such constructions enable the use of optimal transport for signed data, they do not by themselves fully exploit the information contained in both positive and negative signal excursions. This motivates the symmetrized Sinkhorn formulation proposed in this work, which incorporates transport information from both the original signals and their negations. The resulting loss retains the favorable smoothness and transport-based comparison properties of Sinkhorn divergences while providing a more balanced representation of signed oscillatory data.

Building on these ideas, the next section develops the proposed symmetrized Sinkhorn-Gibbs framework for signed oscillatory data.

\section{Construction of the Symmetrized Sinkhorn-Gibbs Posterior}
\label{sec:inf_OT}

In this section, we construct a Gibbsian inference framework for oscillatory inverse problems based on transport-driven loss functions. We begin by introducing a data-processing procedure that enables the application of optimal transport to signed signals while preserving stability and convexity-related of the resulting loss function. We then develop a statistical interpretation of the processed data that naturally integrates transport-based discrepancies into the Gibbsian framework and leads to the construction of a symmetrized Sinkhorn-Gibbs posterior. The section concludes with a practical adaptive MCMC algorithm for posterior computation. The theoretical analysis and formal proofs supporting the proposed framework are presented separately in Section~\ref{sec:analysis}.

\subsection{Data Processing}
\label{sec:data_processing}

Consider a set of raw observations $\{\tilde{g}(\mathbf{x}_i)\}_{i=1}^n$ and raw model predictions $\{\tilde{f}(\boldsymbol{\theta})(\mathbf{x}_i)\}_{i=1}^n$ of the parametric forward map $f:\Theta\rightarrow\mathcal{F}(X)$ evaluated at discrete points $\{\mathbf{x}_i\}_{i=1}^n \subset X \subset \mathbb{R}^{n_d}$. As discussed in Section~\ref{sec:OT_loss}, the application of transport-based divergences to oscillatory inverse problems requires additional care because optimal transport operates on probability measures, whereas the observed data and model predictions are generally signed. We therefore introduce a two-step data-processing procedure that converts the raw signals into probability vectors suitable for transport-based comparisons. Beyond ensuring compatibility with optimal transport, this construction preserves convexity-related properties of the resulting loss function and supports the robustness guarantees established in Section~\ref{sec:analysis}.

\paragraph{I. Mean-centering.} 
Many oscillatory signals contain global offsets that are not directly informative for comparing their underlying structure. We therefore remove such offsets by subtracting the empirical mean of the discrete samples. 
This operation focuses the comparison on deviations from the average signal level, improves robustness to additive shifts, and aligns the discrete data with the zero-mean function space $\mathcal{F}(X)$ considered in the theoretical analysis of Section~\ref{sec:analysis}. 
Given raw measurements $\{\tilde{g}(\mathbf{x}_i)\}_{i=1}^n$ and model predictions $\{\tilde{f}(\boldsymbol{\theta})(\mathbf{x}_i)\}_{i=1}^n$, we define
\[
g(\mathbf{x}_i) := \tilde{g}(\mathbf{x}_i) - \frac{1}{n} \sum_{j=1}^n \tilde{g}(\mathbf{x}_j), \qquad
f(\boldsymbol{\theta})(\mathbf{x}_i) := \tilde{f}(\boldsymbol{\theta})(\mathbf{x}_i) - \frac{1}{n} \sum_{j=1}^n \tilde{f}(\boldsymbol{\theta})(\mathbf{x}_j),
\]
so that the processed discrete signals have zero empirical mean.

\paragraph{II. Mapping to the probability simplex.} 
Mean-centering generally produces signed data, even when the original measurements are strictly positive. Since optimal transport requires positive probability measures, we map the signed, mean-centered data into positive probability vectors using a smooth, strictly increasing normalization function $\sigma: \mathbb{R} \to \mathbb{R}_+$. Specifically, we set
\[
T_\sigma g(\mathbf{x}_i) := \frac{\sigma\big(g(\mathbf{x}_i)\big)}{\sum_{j=1}^n \sigma\big(g(\mathbf{x}_j)\big)}, \qquad
T_\sigma f(\boldsymbol{\theta})(\mathbf{x}_i) := \frac{\sigma\big(f(\boldsymbol{\theta})(\mathbf{x}_i)\big)}{\sum_{j=1}^n \sigma\big(f(\boldsymbol{\theta})(\mathbf{x}_j)\big)},
\]
which yields discrete probability vectors $T_\sigma \mathbf{g}_n$ and $T_\sigma \mathbf{f}_n(\boldsymbol{\theta})$ supported on $\{\mathbf{x}_i\}_{i=1}^n$. As detailed in Section~\ref{sec:analysis}, careful normalization ensures the resulting Sinkhorn loss retains convexity and Lipschitz continuity with respect to perturbations. 
A concrete choice of the normalization function $\sigma$ used in our experiments is the softplus function,
\begin{equation}\label{eq:softplus}
\sigma_{\delta}(z) = \ln\big(1 + e^{\delta z}\big),
\end{equation}
where $\delta > 0$ is a scaling hyperparameter that can be adjusted based on the expected signal amplitude.

The resulting probability vectors serve as the fundamental objects used in the transport-based comparisons that define the Gibbs posteriors developed below.

\subsection{Statistical Interpretation and Discrete Posterior Construction}
\label{sec:stat_interp}

We interpret the processed observed data as a discrete random vector supported on the set of points $\{\mathbf{x}_i\}_{i=1}^n \subset X$, with associated probabilities given by the normalized weights $\{ T_\sigma g(\mathbf{x}_i) \}_{i=1}^n$. This perspective explicitly incorporates the geometric structure of the data, ensuring that spatial relationships encoded in $\{\mathbf{x}_i\}_{i=1}^n$ are preserved throughout the inference process.

This interpretation differs fundamentally from the standard Bayesian approach, where measurements $\{g(\mathbf{x}_i)\}_{i=1}^n$ are treated as realizations of a random variable evaluated at fixed inputs $\{\mathbf{x}_i\}_{i=1}^n$, or more generally, where the pair $(\mathbf{x}, g(\mathbf{x}))$ is modeled as a random vector producing $n$ i.i.d. samples $\{(\mathbf{x}_i, g(\mathbf{x}_i))\}_{i=1}^n$. In contrast, by treating the processed data as a discrete probability measure on $\{\mathbf{x}_i\}_{i=1}^n$, our approach enables a geometrically meaningful integration of Sinkhorn-based divergences into the Gibbsian framework.


The processed observed data and model predictions induce discrete probability measures supported on the grid points $\{\mathbf{x}_i\}_{i=1}^n$. The corresponding normalized probability vectors are used to define the \emph{normalized discrete Sinkhorn divergence}
\begin{equation}\label{eq:Sinkhorn_discrete}
S_{\sigma}(\mathbf{f}_n(\boldsymbol{\theta}), \mathbf{g}_n) := S\left(T_\sigma \mathbf{f}_n(\boldsymbol{\theta}), \, T_\sigma \mathbf{g}_n \right),
\end{equation}
where $S(\cdot,\cdot)$ denotes the discrete Sinkhorn divergence \eqref{S_unbiased} applied to the normalized probability vectors formed from the processed data samples. We emphasize that $S_{\sigma}$ is applied directly to the vectors
\[
\mathbf{g}_n := \{g(\mathbf{x}_i)\}_{i=1}^n, \qquad \mathbf{f}_n(\boldsymbol{\theta}) := \{f(\boldsymbol{\theta})(\mathbf{x}_i)\}_{i=1}^n,
\]
which have zero empirical mean as ensured by the data processing procedure of Section~\ref{sec:data_processing}.

The empirical risk function $\Phi(\boldsymbol{\theta}; \mathbf{g}_n)$ obtained from the square of the normalized discrete Sinkhorn divergence \eqref{eq:Sinkhorn_discrete} provides a natural transport-based loss within the Gibbsian framework. This leads to a normalized Sinkhorn-Gibbs posterior
\begin{equation} \label{DSloss_normalized}
\Pi_n(d\boldsymbol{\theta}) = \frac{\exp\left(- \lambda n \, \Phi(\boldsymbol{\theta}; \mathbf{g}_n)\right) \, \Pi_0(d\boldsymbol{\theta})}{\int_\Theta \exp\left(- \lambda n \, \Phi(\boldsymbol{\theta}; \mathbf{g}_n)\right) \, \Pi_0(d\boldsymbol{\theta})}, \qquad
\Phi(\boldsymbol{\theta}; \mathbf{g}_n) = S_{\sigma}(\mathbf{f}_n(\boldsymbol{\theta}),\mathbf{g}_n)^2.
\end{equation}

While this construction successfully incorporates transport geometry into the Gibbsian framework, the transformation of signed oscillatory signals into probability measures may not fully exploit information contained in both positive and negative signal excursions. To incorporate transport information from both phases of the oscillation, we introduce the \emph{symmetrized normalized squared Sinkhorn loss}
\begin{equation}\label{eq:symmetrized_squared_Sink}
    D_\sigma(\mathbf{f}_n(\boldsymbol{\theta}), \mathbf{g}_n)^2 := S_\sigma(\mathbf{f}_n(\boldsymbol{\theta}), \mathbf{g}_n)^2 + S_\sigma(-\mathbf{f}_n(\boldsymbol{\theta}), -\mathbf{g}_n)^2.
\end{equation}
Using the symmetrized loss as the empirical risk yields the symmetrized Sinkhorn-Gibbs posterior
\begin{equation} \label{DSloss_normalized_symmetrized}
\Pi_n(d\boldsymbol{\theta}) = \frac{\exp\left(- \lambda n \, \Phi(\boldsymbol{\theta}; \mathbf{g}_n)\right) \, \Pi_0(d\boldsymbol{\theta})}{\int_\Theta \exp\left(- \lambda n \, \Phi(\boldsymbol{\theta}; \mathbf{g}_n)\right) \, \Pi_0(d\boldsymbol{\theta})}, \qquad
\Phi(\boldsymbol{\theta}; \mathbf{g}_n) = D_{\sigma}(\mathbf{f}_n(\boldsymbol{\theta}),\mathbf{g}_n)^2.
\end{equation}

Beyond incorporating information from both phases of the oscillation, the symmetrized loss is designed to reduce spurious oscillations and competing local minima in the empirical risk landscape and to enhance robustness to sign-sensitive perturbations in oscillatory data. As discussed in Section~\ref{sec:analysis}, the symmetrized construction admits stronger convexity and robustness properties than its non-symmetrized counterpart under suitable assumptions. Moreover, the numerical experiments of Section~\ref{sec:numerical_exp} consistently indicate improved posterior performance when the symmetrized loss is employed. The symmetrized Sinkhorn-Gibbs posterior \eqref{DSloss_normalized_symmetrized} therefore constitutes the primary inference framework proposed in this work.

\subsection{Adaptive Sampling of Gibbs Posteriors with Learned Inverse Temperature}
\label{sec:adaptive_mcmc}

As Gibbs posteriors retain the same mathematical structure as Bayesian posteriors, many standard Bayesian sampling algorithms can be adapted to compute them. A key practical challenge, however, is the specification of the inverse-temperature parameter \(\lambda\), which controls the concentration of the posterior around minimizers of the empirical risk. The choice of \(\lambda\) can have a substantial effect on posterior uncertainty and often requires careful calibration. Several approaches have been proposed in the literature, including coverage-based calibration, posterior predictive checks, and fully Bayesian treatments that place a prior distribution on the inverse temperature itself; see, for example, \cite{SyringMartin2019,ZafarNicholls2024,LeeLiuNicholls2025}.

In this subsection, we present a general sampling framework in which the inverse-temperature parameter is treated as an unknown quantity and inferred jointly with the physical model parameters. This formulation provides a flexible approach that naturally accommodates discrepancies with different numerical scales and allows the data to determine the appropriate posterior concentration. While some of the numerical experiments in Section~\ref{sec:numerical_exp} employ calibrated values of $\lambda$ for comparison purposes, the framework presented here applies to the more general setting in which $\lambda$ is learned as part of the inference procedure.

Importantly, the adaptive sampling strategy developed below is independent of the specific empirical risk used to construct the Gibbs posterior. In the numerical experiments of Section~\ref{sec:numerical_exp}, we consider Gibbs posteriors generated by the squared Euclidean discrepancy, a trace-Wasserstein discrepancy, and the proposed symmetrized Sinkhorn loss. The sampling framework presented here applies uniformly to all of these cases.

Given observed data $\mathbf{g}_n$ and an empirical risk function
$\Phi(\boldsymbol{\theta};\mathbf{g}_n)$ defining a Gibbs posterior of the form
\eqref{Gibbs_posterior_discrete}, we place a prior distribution on the inverse-temperature parameter $\lambda$ and consider the corresponding joint Gibbs posterior on $\Theta\times(0,\infty)$,
\[
\Pi_n(d\boldsymbol{\theta},d\lambda)
\propto
\exp\!\left[
-\lambda n \Phi(\boldsymbol{\theta};\mathbf{g}_n)
\right]
\Pi_0(d\boldsymbol{\theta})
\Pi_0(d\lambda).
\]
Assuming that the prior measures admit densities $\pi_0(\boldsymbol{\theta})$ and $\pi_0(\lambda)$ with respect to suitable reference measures, the corresponding joint posterior density is given by
\[
\pi(\boldsymbol{\theta},\lambda\mid \mathbf{g}_n)
\propto
\exp\!\left[-\lambda n \Phi(\boldsymbol{\theta};\mathbf{g}_n)\right]
\pi_0(\boldsymbol{\theta})\pi_0(\lambda).
\]
To preserve positivity and facilitate efficient exploration across multiple orders of magnitude, the inverse temperature is updated through its logarithm,
\[
\ell=\log\lambda.
\]

To sample from the joint posterior, we employ an adaptive Metropolis-within-Gibbs algorithm. At each iteration, we first update \(\boldsymbol{\theta}\) while keeping \(\ell\) fixed, and then update \(\ell\) while keeping \(\boldsymbol{\theta}\) fixed. The proposal for \(\boldsymbol{\theta}\in\mathbb{R}^{n_p}\) is a multivariate Gaussian random walk,
\[
\boldsymbol{\theta}^{*}
=
\boldsymbol{\theta}^{(k-1)}
+
\exp\!\left(\eta_{\theta}^{(k-1)}\right)
L_{\theta}^{(k-1)} z,
\qquad
z\sim N(0,I_{n_p}),
\]
where
\[
L_{\theta}^{(k-1)}
\bigl(L_{\theta}^{(k-1)}\bigr)^\top
=
\Sigma_{\theta}^{(k-1)}.
\]
Here, \(\Sigma_{\theta}^{(k-1)}\) denotes an adaptive empirical covariance matrix and \(\eta_{\theta}^{(k-1)}\) is an adaptive log-scaling parameter controlling the proposal magnitude. Similarly, the proposal for the inverse temperature is defined in log-space by
\[
\ell^{*}
=
\ell^{(k-1)}
+
\exp\!\left(\eta_{\lambda}^{(k-1)}\right)
\sqrt{v_{\lambda}^{(k-1)}}\,\xi,
\qquad
\xi\sim N(0,1),
\]
where \(v_{\lambda}^{(k-1)}\) denotes the adaptive empirical variance of the log-temperature chain and \(\eta_{\lambda}^{(k-1)}\) is an adaptive log-scaling parameter.

The adaptation follows the ASM+AM strategy of Vihola~\cite{vihola2020}, which combines adaptive Metropolis (AM) covariance estimation with adaptive scaling Metropolis (ASM) proposal tuning. The AM component updates the empirical covariance of the \(\boldsymbol{\theta}\)-chain according to
\begin{equation}\label{eq:adaptive_cov}
\Sigma_{\theta}^{(k)}
=
\Sigma_{\theta}^{(k-1)}
+
\gamma_k
\left(
\delta_{\theta}^{(k)}
(\delta_{\theta}^{(k)})^\top
-
\Sigma_{\theta}^{(k-1)}
\right),
\qquad
\delta_{\theta}^{(k)}
=
\boldsymbol{\theta}^{(k)}
-
\mu_{\theta}^{(k-1)},
\end{equation}
where the empirical mean is updated by
\[
\mu_{\theta}^{(k)}
=
\mu_{\theta}^{(k-1)}
+
\gamma_k
\delta_{\theta}^{(k)}.
\]
The proposal scale is adapted through the ASM recursion
\begin{equation}\label{eq:adaptive_log_scale}
\eta_{\theta}^{(k)}
=
\eta_{\theta}^{(k-1)}
+
\gamma_k
\left(
\alpha_{\theta}^{(k)}
-
\alpha_{\theta}^{\rm tgt}
\right),
\qquad
\gamma_k=(k+1)^{-2/3},
\end{equation}
where \(\alpha_{\theta}^{(k)}\) denotes the Metropolis acceptance probability associated with the current \(\boldsymbol{\theta}\)-proposal. Following standard practice, we use the target acceptance probability
\[
\alpha_{\theta}^{\rm tgt}
=
\begin{cases}
0.44, & n_p=1,\\
0.234, & n_p\ge 2.
\end{cases}
\]

The same AM+ASM construction is applied to the log-temperature update. Specifically, the empirical variance \(v_{\lambda}^{(k)}\) is updated analogously to \eqref{eq:adaptive_cov}, while the corresponding log-scaling parameter \(\eta_{\lambda}^{(k)}\) is updated according to \eqref{eq:adaptive_log_scale} with target acceptance probability \(\alpha_{\lambda}^{\rm tgt}=0.44\).

The complete adaptive sampling procedure is summarized in Algorithm~\ref{alg:adaptive_mhg}. In the implementation, covariance eigenvalues, empirical variances, and log-scaling parameters are constrained to remain within prescribed bounds in order to prevent degenerate proposal distributions during the early adaptation phase.

\begin{algorithm}[!ht]
\caption{{\fontsize{11}{12}\selectfont
Adaptive sampler for Gibbs posteriors with learned inverse temperature}}
\label{alg:adaptive_mhg}
\begin{algorithmic}
\medskip

\STATE {\bf 1.} {\it Input}: \ \ \ \ $\mathbf{g}_n=\{g(\mathbf{x}_i)\}_{i=1}^n$: observed data on grid points $\{\mathbf{x}_i\}_{i=1}^n\subset X\subset\mathbb{R}^{n_d}$.

\hskip 2.14cm $\mathbf{f}_n(\boldsymbol{\theta})=\{f(\boldsymbol{\theta})(\mathbf{x}_i)\}_{i=1}^n$: model predictions for $\boldsymbol{\theta}\in\Theta\subset\mathbb{R}^{n_p}$.

\hskip 2.14cm $\Phi(\boldsymbol{\theta};\mathbf{g}_n)
=L(\mathbf{f}_n(\boldsymbol{\theta}),\mathbf{g}_n)$: empirical risk.

\hskip 2.14cm $\pi_0(\boldsymbol{\theta})$ and $\pi_0(\ell)$: prior densities, with $\ell=\log\lambda$ and $\lambda=\exp(\ell)>0$.

\hskip 2.14cm $K$: total number of MCMC iterations.

\hskip 2.14cm $\alpha_{\theta}^{\rm tgt}$ and $\alpha_{\lambda}^{\rm tgt}$: target acceptance probabilities.

\medskip
\STATE {\bf 2.} {\it Initialization}:

\ \ \ \ \  Choose $(\boldsymbol{\theta}^{(0)},\ell^{(0)})$ and set $\lambda^{(0)}=\exp(\ell^{(0)})$.

\ \ \ \ \ Evaluate $\Phi^{(0)}=\Phi(\boldsymbol{\theta}^{(0)};\mathbf{g}_n)$.

\ \ \ \ \ Initialize $\mu_\theta^{(0)} = \boldsymbol{\theta}^{(0)}$, $\Sigma_\theta^{(0)}$, $\mu_\lambda^{(0)} = \ell^{(0)}$, $v_\lambda^{(0)}$, and $\eta_\theta^{(0)} = \eta_\lambda^{(0)} = 0$.

\medskip
\STATE {\bf 3.} {\it Adaptive Gibbs posterior sampling loop}:

\medskip
\ \ \  {\bf for} $k=1,2,\ldots,K$

\begin{itemize}[leftmargin=1.0cm]
\setlength\itemsep{-0.35em}

\item Metropolis step 1: update $\boldsymbol{\theta}$, given $\lambda^{(k-1)}$.

\begin{itemize}[leftmargin=.05cm]
\setlength\itemsep{-0.35em}

\item[$\circ$] Draw $z\sim N(0,I_{n_p})$ and let $\boldsymbol{\theta}^*
=
\boldsymbol{\theta}^{(k-1)}
+
\exp(\eta_\theta^{(k-1)})L_\theta^{(k-1)}z$, with $L_\theta^{(k-1)} (L_\theta^{(k-1)})^\top=\Sigma_\theta^{(k-1)}$.

\item[$\circ$] Evaluate $\Phi^*=\Phi(\boldsymbol{\theta}^*;\mathbf{g}_n)$ and compute $\alpha_\theta^{(k)}
=
\min\left\{
1,\,
\frac{
\exp[-\lambda^{(k-1)}n\Phi^*]\pi_0(\boldsymbol{\theta}^*)
}{
\exp[-\lambda^{(k-1)}n\Phi^{(k-1)}]\pi_0(\boldsymbol{\theta}^{(k-1)})
}
\right\}$.

\item[$\circ$] Set $(\boldsymbol{\theta}^{(k)},\Phi^{(k)})=(\boldsymbol{\theta}^*,\Phi^*)$ with probability $\alpha_\theta^{(k)}$; otherwise $(\boldsymbol{\theta}^{(k)},\Phi^{(k)})=(\boldsymbol{\theta}^{(k-1)},\Phi^{(k-1)})$.

\end{itemize}

\item Metropolis step 2: update $\lambda\mid\boldsymbol{\theta}^{(k)},\mathbf{g}_n$ in log-space.

\begin{itemize}[leftmargin=.05cm]
\setlength\itemsep{-0.25em}

\item[$\circ$] Draw $\xi\sim N(0,1)$ and propose $\ell^*
=
\ell^{(k-1)}
+
\exp(\eta_\lambda^{(k-1)})\sqrt{v_\lambda^{(k-1)}}\,\xi$, and set $\lambda^*=\exp(\ell^*)$.

\item[$\circ$] Compute 
$\alpha_\lambda^{(k)}
=
\min\left\{
1,\,
\frac{
\exp[-\lambda^*n\Phi^{(k)}]\pi_0(\ell^*)
}{
\exp[-\lambda^{(k-1)}n\Phi^{(k)}]\pi_0(\ell^{(k-1)})
}
\right\}$.

\item[$\circ$] Set $\ell^{(k)}=\ell^*$ with probability $\alpha_\lambda^{(k)}$; otherwise $\ell^{(k)}=\ell^{(k-1)}$. Then set $\lambda^{(k)}=\exp(\ell^{(k)})$.

\end{itemize}

\item Adapt proposal statistics.

\begin{itemize}[leftmargin=.05cm]
\setlength\itemsep{-0.25em}

\item[$\circ$] Update $\mu_\theta^{(k)}$ and $\Sigma_\theta^{(k)}$ by \eqref{eq:adaptive_cov}, and similarly update $\mu_\lambda^{(k)}$ and $v_\lambda^{(k)}$.

\item[$\circ$] Update the adaptive log-scaling $\eta_\theta^{(k)}$ by \eqref{eq:adaptive_log_scale}, and similarly update $\eta_\lambda^{(k)}$.

\item[$\circ$] Clamp $\Sigma_\theta^{(k)}$, $v_\lambda^{(k)}$, $\eta_\theta^{(k)}$, and $\eta_\lambda^{(k)}$ within prescribed bounds if necessary.

\end{itemize}

\end{itemize}

\ \ \ {\bf end for}

\medskip
\STATE {\bf 4.} {\it Output}: return
$\{(\boldsymbol{\theta}^{(k)},\lambda^{(k)})\}_{k=1}^K$
after discarding burn-in samples and thinning the chain.

\end{algorithmic}
\end{algorithm}


\section{Theoretical Analysis of Symmetrized Sinkhorn-Gibbs Inference}
\label{sec:analysis}

In this section, we provide a rigorous analysis of the symmetrized Sinkhorn-Gibbs inference framework introduced in Section~\ref{sec:inf_OT}. Our analysis considers the observed data $g$ and model predictions $f(\boldsymbol{\theta})$ as bounded, real-valued functions defined on a compact domain $X$, and formalizes the transformations needed to integrate them into optimal transport-based inference. 
The section is organized into three parts: first, we establish the smoothness, convexity, and Lipschitz continuity properties underlying the normalized Sinkhorn divergence and the associated symmetrized loss; second, we prove the well-definedness of the resulting Gibbs posterior; and third, we demonstrate robustness of the posterior with respect to perturbations in the data and the forward model.

\subsection{Convexity and Stability of Symmetrized Sinkhorn Losses}
\label{sec:convexity}

A key challenge in integrating data into the Sinkhorn-Gibbs framework is ensuring that data transformations preserve the smoothness and convexity of the Sinkhorn-based divergences without loss of information. It is well known that the convexity of the Wasserstein distance can be destroyed by data normalization \cite{FrenchOT3,Yang_Engquist:18,Engquist_Yang:19b,Motamed_Appelo:19}. We show that employing particular smooth, one-to-one normalization functions preserves smoothness and asymptotic convexity in our setting (see Propositions~\ref{thm:smoothness}-\ref{thm:symmetrized_convexity} and Corollary~\ref{corollary_1}).

A second critical requirement is the preservation of stability of the Sinkhorn divergence. In particular, we aim for a strong notion of stability measured in the negative Sobolev space $\dot{H}^{-1}(X)$. For any bounded function $h$ defined on $X$, the $\dot{H}^{-1}$ semi-norm is given by
\[
\| h \|_{\dot{H}^{-1}(X)} := \sup_{\phi \in \dot{H}^1(X)} \left\{ \int_X h \, \phi \, d\mathbf{x} \, \middle| \, \|\phi\|_{\dot{H}^1(X)}^2 = \int_X |\nabla \phi|^2 \, d\mathbf{x} \le 1 \right\},
\]
which measures the size of $h$ in a weak sense, making the semi-norm insensitive to high-frequency oscillations. 
This semi-norm is only finite if $h$ has zero mean over $X$: if $\int_X h(\mathbf{x}) \, d\mathbf{x} \neq 0$, constant test functions $\phi$ can drive the supremum arbitrarily large. Therefore, before normalizing the data, we ensure that both the observed and simulated signals satisfy
\[
\int_X g(\mathbf{x}) \, d\mathbf{x} = 0, \qquad \int_X f(\boldsymbol{\theta})(\mathbf{x}) \, d\mathbf{x} = 0,
\]
so that $g, f(\boldsymbol{\theta}) \in \mathcal{F}(X)$, the space of bounded, zero-mean functions defined in \eqref{FX1}. This mean-zero condition guarantees that the $\dot{H}^{-1}$ semi-norm is well-defined and enables stability results such as Lipschitz continuity of the Sinkhorn divergence (see Proposition~\ref{thm:Lipschitz}).

Once the signals are mean-centered, they must be transformed into probability densities to apply optimal transport. This is achieved by mapping the processed, signed data to positive densities via a smooth, strictly increasing normalization function $\sigma: \mathbb{R} \to \mathbb{R}_+$, such as the softplus function defined in \eqref{eq:softplus}. The normalized densities are given by
\begin{equation} \label{eq:T_sigma_cont}
T_\sigma g(\mathbf{x}) := \frac{\sigma(g(\mathbf{x}))}{\int_X \sigma(g(\mathbf{x}')) \, d\mathbf{x}'}, 
\qquad
T_\sigma f(\boldsymbol{\theta})(\mathbf{x}) := \frac{\sigma(f(\boldsymbol{\theta})(\mathbf{x}))}{\int_X \sigma(f(\boldsymbol{\theta})(\mathbf{x}')) \, d\mathbf{x}'}.
\end{equation}

These transformations ensure compatibility with optimal transport and enable the definition of the \emph{normalized Sinkhorn divergence},
\begin{equation}\label{eq:Sinkhorn_density}
S_{\sigma,\varepsilon}(f(\boldsymbol{\theta}), g) := S_\varepsilon\big(T_\sigma f(\boldsymbol{\theta}), \, T_\sigma g \big),
\end{equation}
where $S_{\varepsilon}(\cdot,\cdot)$ denotes the Sinkhorn divergence \eqref{eq:Sinkhorn_continuous} applied to the probability measures induced by the normalized densities in \eqref{eq:T_sigma_cont}. We emphasize that $S_{\sigma,\varepsilon}$ quantifies the dissimilarity directly between raw mean-zero functions $f(\boldsymbol{\theta}), g \in \mathcal{F}(X)$ by mapping them into the optimal transport framework.

The primary loss function considered in this work is the symmetrized normalized squared Sinkhorn loss
\begin{equation}\label{eq:sym_loss_cont}
\mathcal D_{\sigma,\varepsilon}^2(f,g)
:=
S_{\sigma,\varepsilon}^2(f,g)
+
S_{\sigma,\varepsilon}^2(-f,-g),
\end{equation}
which corresponds to the continuous analogue of the discrete loss introduced in Section~\ref{sec:stat_interp}. The results below first establish fundamental properties of the normalized Sinkhorn divergence and then show how these properties extend naturally to the symmetrized loss. 
Specifically, we show that the proposed transformations, while ensuring compatibility with optimal transport, preserve essential mathematical properties of the normalized Sinkhorn divergence and the associated symmetrized loss, including smoothness, asymptotic convexity, and Lipschitz continuity in the $\dot{H}^{-1}$ norm. These foundational results provide theoretical support for the symmetrized Sinkhorn-Gibbs framework developed later in this section.

\begin{proposition}\label{thm:smoothness}
Let $f, g \in \mathcal{F}(X)$. The normalized Sinkhorn divergence $S_{\sigma,\varepsilon}(f,g)$, defined in \eqref{eq:Sinkhorn_density} with a one-to-one, smooth normalization function $\sigma: \mathbb{R} \to \mathbb{R}_+$, is a symmetric, positive definite, and smooth loss function with respect to $f$ and $g$.
\end{proposition}

\begin{proof}
This follows directly from the properties of the Sinkhorn divergence $S_\varepsilon$, as established in Theorem 1 of \cite{Feydy_etal:19}, together with the injectivity and smoothness of the normalization function $\sigma$. These ensure that the compositions $T_\sigma f$ and $T_\sigma g$ depend smoothly on $f$ and $g$, and the positive definiteness and symmetry of $S_\varepsilon$ carry over.
\end{proof}

\begin{corollary}\label{corollary_1}
Let $f, g \in \mathcal{F}(X)$. 
The symmetrized normalized squared Sinkhorn loss $\mathcal D_{\sigma,\varepsilon}^2(f,g)$, defined in \eqref{eq:sym_loss_cont} with a one-to-one, smooth normalization function $\sigma: \mathbb{R} \to \mathbb{R}_+$, is symmetric, positive definite, and smooth with respect to $f$ and $g$.
\end{corollary}

\begin{proof}
Immediate from Proposition~\ref{thm:smoothness}, since
$\mathcal D_{\sigma,\varepsilon}^2$
is the sum of two smooth positive-definite terms.    
\end{proof}

\begin{proposition}\label{thm:asymptotic_convexity}
Let $f, g \in \mathcal{F}(X)$ and consider the softplus normalization $\sigma_\delta$ defined in \eqref{eq:softplus}. Then, as $\delta \to \infty$, the normalized Sinkhorn divergence $S_{\sigma_\delta,\varepsilon}(f,g)$ converges pointwise to
\[
S_\varepsilon\left(\frac{f^+}{\int_X f^+}, \, \frac{g^+}{\int_X g^+}\right),
\]
where $f^+(\mathbf{x}) = \max\{f(\mathbf{x}),0\}$. 
The limiting Sinkhorn divergence on these normalized positive parts of $f$ and $g$ is a well-defined, convex function of its arguments.
\end{proposition}

\begin{proof}
First, note that for $\delta \to \infty$, the softplus scaling satisfies \(\sigma_\delta(f(\mathbf{x})) = \ln(1 + e^{\delta f(\mathbf{x})}) \rightarrow \delta f(\mathbf{x})\) for $f(\mathbf{x})>0$, while for $f(\mathbf{x})<0$ the term converges to 0. We hence have:
\[
\lim_{\delta \to \infty} \sigma_\delta(f(\mathbf{x})) = 
\begin{cases}
0, & f^+(\mathbf{x})=0, \\[6pt]
\delta f^+(\mathbf{x}), & f^+(\mathbf{x})>0,
\end{cases}
\]
and the normalization integral behaves as
\[
\lim_{\delta \to \infty} \int_X \sigma_\delta(f(\mathbf{x}')) \, d\mathbf{x}' = \delta \int_X f^+(\mathbf{x}') \, d\mathbf{x}'.
\]
Therefore, the normalized density converges pointwise to
\[
\lim_{\delta \to \infty} T_{\sigma_{\delta}} f(\mathbf{x}) = \frac{f^+(\mathbf{x})}{\int_X f^+(\mathbf{x}') \, d\mathbf{x}'}.
\]
A similar argument gives
\[
\lim_{\delta \to \infty} T_{\sigma_{\delta}} g(\mathbf{x}) = \frac{g^+(\mathbf{x})}{\int_X g^+(\mathbf{x}') \, d\mathbf{x}'}.
\]
Consequently, the normalized Sinkhorn divergence converges to
\[
\lim_{\delta \to \infty} S_{\sigma_{\delta},\varepsilon}(f,g) = S_{\varepsilon}\left(\frac{f^+}{\int_X f^+}, \, \frac{g^+}{\int_X g^+}\right).
\]
Convexity of the limiting Sinkhorn divergence in $(f^+, g^+)$ follows from Theorem 1 of \cite{Feydy_etal:19} applied to these probability densities, implying the asymptotic convexity of the normalized Sinkhorn divergence.
\end{proof}

\begin{proposition}\label{thm:symmetrized_convexity}
Let $f, g \in \mathcal{F}(X)$ and consider the softplus normalization $\sigma_\delta$ defined in \eqref{eq:softplus}. Then, as $\delta \to \infty$, the symmetrized normalized squared Sinkhorn loss $\mathcal D_{\sigma_\delta,\varepsilon}^2(f,g)$ converges pointwise to
\[
S_\varepsilon^2\left(\frac{f^+}{\int_X f^+}, \, \frac{g^+}{\int_X g^+}\right) + S_\varepsilon^2\left(\frac{f^-}{\int_X f^-}, \, \frac{g^-}{\int_X g^-}\right),
\]
where  $f^+(\mathbf{x}) = \max\{f(\mathbf{x}),0\}$ and $f^-(\mathbf{x}) := \max\{-f(\mathbf{x}),0\}$. 
The limiting loss is a well-defined, convex function of its arguments.
\end{proposition}
\begin{proof}
Applying Proposition~\ref{thm:asymptotic_convexity} to the pair $(f,g)$ yields the asymptotic limit associated with the positive parts $f^+$ and $g^+$. Applying the same proposition to the pair $(-f,-g)$ yields the corresponding limit involving the negative parts $f^-$ and $g^-$. Summing the two limits gives the stated expression. 
Since each limiting Sinkhorn term is convex in its arguments by Proposition~\ref{thm:asymptotic_convexity}, and the sum of convex functions remains convex, the limiting symmetrized loss is also convex.    
\end{proof}

\begin{remark}[Asymptotic Convexity and Practical Implications]
Proposition~\ref{thm:asymptotic_convexity} and Proposition~\ref{thm:symmetrized_convexity} 
establish asymptotic convexity only in the limit $\delta\to\infty$. In particular, the normalized Sinkhorn divergence and the symmetrized loss converge pointwise to convex limiting functions defined through the positive and negative components of the signals. However, uniform convergence of the second derivatives with respect to $f$ and $g$ may fail due to potential discontinuities in the derivatives of $T_{\sigma_\delta}f$ and $T_{\sigma_\delta}g$ near the zeros of $f$ or $g$. Consequently, uniform convexity cannot generally be guaranteed for finite values of $\delta$. 
Nevertheless, Proposition~\ref{thm:symmetrized_convexity} provides theoretical support for the symmetrized loss introduced in Section~\ref{sec:inf_OT}. By incorporating transport information from both positive and negative signal excursions, the symmetrized construction inherits the asymptotic convexity of its constituent Sinkhorn terms while retaining sensitivity to sign-dependent features of oscillatory data. The numerical experiments of Section~\ref{sec:numerical_exp} further suggest substantially improved loss geometry even for moderate values of $\delta$.
\end{remark}

We now turn to another key property of the normalized Sinkhorn divergence: its Lipschitz continuity with respect to perturbations measured in the negative Sobolev norm.

\begin{proposition}\label{thm:Lipschitz}
Let $f, f' \in \mathcal{F}(X)$, and let $\sigma: \mathbb{R} \to \mathbb{R}_+$ be a smooth, strictly increasing normalization function. Then there exist positive constants $M_1, M_2 > 0$, depending only on $\|f\|_{L^\infty(X)}$, $\|f'\|_{L^\infty(X)}$, and $\varepsilon >0$ (with $M_1$, $M_2$ remaining bounded for $\varepsilon$ in any fixed interval $[\varepsilon_{\min}, \varepsilon_{\max}]$), such that
\[
S_{\sigma, \varepsilon}(f, f') = S_{\varepsilon}(T_{\sigma} f, T_{\sigma} f') \le M_1 \, \| T_{\sigma} f - T_{\sigma} f' \|_{\dot{H}^{-1}(X)} \le M_2 \, \| f - f' \|_{\dot{H}^{-1}(X)},
\]
where $S_{\sigma, \varepsilon}(f, f')$ denotes the normalized Sinkhorn divergence defined in \eqref{eq:Sinkhorn_density}.
\end{proposition}

\begin{proof}
The first inequality follows from comparing the Sinkhorn divergence $S_{\varepsilon}$ with the quadratic Wasserstein distance $W_2$ between the normalized densities $\rho = T_\sigma f$ and $\rho' = T_\sigma f'$. Specifically, because the Sinkhorn divergence with entropic regularization $\varepsilon$ converges to the quadratic Wasserstein distance $W_2$ as $\varepsilon \to 0$, we have
\[
\lim_{\varepsilon \to 0} S_\varepsilon(\rho, \rho') = W_2(\rho, \rho'),
\]
uniformly over compactly supported, bounded densities. 
On bounded domains $X$, adding finite entropic regularization $\varepsilon > 0$ introduces a smoothing of the transport plan, but does not cause unbounded deviation from $W_2$. Concretely, Sinkhorn’s entropic regularization tends to reduce transport costs relative to the unregularized problem and adds a bias correction term to ensure that $S_\varepsilon(\rho, \rho) = 0$, preserving its interpretation as a true divergence. Therefore, there exists $C_\varepsilon \ge 1$, depending on $\varepsilon$ and the compactness of $X$, such that
\[
S_\varepsilon(\rho, \rho') \le C_\varepsilon W_2(\rho, \rho'),
\]
with $C_\varepsilon$ bounded for $\varepsilon$ in any fixed interval $[\varepsilon_{\min}, \varepsilon_{\max}]$.

On the other hand, since $W_2$ metrizes weak convergence and is equivalent (up to constants depending on bounds on densities) to the negative Sobolev norm $\dot{H}^{-1}$ on compact domains \cite{Peyre:2018}, we have
\[
W_2(\rho, \rho') \le C' \, \|\rho - \rho'\|_{\dot{H}^{-1}(X)},
\]
where $C'$ depends only on the diameter of $X$ and the essential bounds on $\rho, \rho'$. Composing the two gives the first inequality:
\[
S_\varepsilon(\rho, \rho') \le C_\varepsilon C' \, \|\rho - \rho'\|_{\dot{H}^{-1}(X)},
\]
establishing the constant $M_1 = C_\varepsilon C'$.

Finally, since the normalization operator $T_\sigma$ is locally Lipschitz in $L^\infty$ (due to the smoothness of $\sigma$), and by the definition of the $\dot{H}^{-1}$ semi-norm as the dual to $\dot{H}^1$, there exists a constant $L_\sigma > 0$ depending on the $L^\infty$ norms of $f, f'$ such that
\[
\|T_\sigma f - T_\sigma f'\|_{\dot{H}^{-1}(X)} \le L_\sigma \, \| f - f' \|_{\dot{H}^{-1}(X)},
\]
which yields the second inequality with $M_2 = M_1 L_\sigma$.
\end{proof}

\subsection{Well-Definedness of Symmetrized Sinkhorn-Gibbs Posteriors}
\label{sec:welldefinedness}

We now establish that the continuous symmetrized Sinkhorn-Gibbs posterior defines a valid probability measure on the compact parameter space $\Theta \subset \mathbb{R}^{n_p}$, ensuring the well-definedness of the proposed inference framework.

We first restate the continuous Gibbs posterior,
\begin{equation} \label{Gibbs_posterior_continuous2}
\Pi(d\boldsymbol{\theta}) =
\frac{
\exp\left(- \lambda \, \Phi(\boldsymbol{\theta}; g)\right)
\, \Pi_0(d\boldsymbol{\theta})
}{
\int_\Theta
\exp\left(- \lambda \, \Phi(\boldsymbol{\theta}; g)\right)
\, \Pi_0(d\boldsymbol{\theta})
},
\end{equation}
where the empirical risk functional is given by
\begin{equation}\label{eq:Sink_risk_cont}
\Phi(\boldsymbol{\theta}; g)
=
\mathcal D_{\sigma,\varepsilon}^2
\bigl(
f(\boldsymbol{\theta}),
g
\bigr),
\end{equation}
with $\mathcal D_{\sigma,\varepsilon}^2$ denoting the symmetrized normalized squared Sinkhorn loss introduced in \eqref{eq:sym_loss_cont}, and where the forward map $f:\Theta\to\mathcal F(X)$ assigns each parameter $\boldsymbol{\theta}$ to a function $f(\boldsymbol{\theta})$ on $X$. 
We note that together,
\eqref{Gibbs_posterior_continuous2}--\eqref{eq:Sink_risk_cont}
represent the continuous analogue of the symmetrized discrete posterior
\eqref{DSloss_normalized_symmetrized}
introduced in Section~\ref{sec:stat_interp}.

Since the Gibbs posterior \eqref{Gibbs_posterior_continuous2} mirrors the structure of Bayes' theorem, its well-definedness can be established by adapting standard arguments from Bayesian inverse problems \cite{Stuart:10,Dashti_Stuart:16,Sullivan:2017}.

\begin{theorem}[Well-definedness]
\label{well-defined}
Suppose $\Theta \subset \mathbb{R}^{n_p}$ is compact, and let $\Pi_0$ be a Borel probability measure on $\Theta$. Assume the forward map $f:\Theta\to\mathcal F(X)$ is continuous when $\mathcal F(X)$ is equipped with the $L^\infty$ norm topology. Then the symmetrized Sinkhorn risk $\Phi$ defined in \eqref{eq:Sink_risk_cont} satisfies:
\begin{itemize}
    \item[(i)]
    $\Phi(\cdot; g)$ is measurable on $\Theta$ for every
    $g \in \mathcal F(X)$;

    \item[(ii)]
    $\Phi(\boldsymbol{\theta}; g) \ge 0$
    for all
    $\boldsymbol{\theta} \in \Theta$
    and
    $g \in \mathcal F(X)$;

    \item[(iii)]
    There exists a constant
    $0 < C < \infty$,
    depending only on uniform bounds of
    $f$
    and
    $g$,
    such that
    $\Phi(\boldsymbol{\theta}; g) \le C$
    for all
    $\boldsymbol{\theta} \in \Theta$
    and
    $g \in \mathcal F(X)$.
\end{itemize}
Consequently, the symmetrized Sinkhorn-Gibbs posterior $\Pi$ defined by
\eqref{Gibbs_posterior_continuous2}--\eqref{eq:Sink_risk_cont}
is a well-defined Borel probability measure on $\Theta$.
\end{theorem}

\begin{proof}
Non-negativity (ii) follows directly from the definition of the symmetrized loss
$\mathcal D_{\sigma,\varepsilon}^2$
and the fact that
$S_{\sigma,\varepsilon}^2 \ge 0$. 
By the continuity of the forward map $f$ and Corollary~\ref{corollary_1}, the symmetrized loss $\mathcal D_{\sigma,\varepsilon}^2
\bigl(f(\boldsymbol{\theta}),g\bigr)$ depends continuously on
$\boldsymbol{\theta}$
for each fixed
$g \in \mathcal F(X)$.
Hence the risk functional
$\Phi(\cdot;g)$
is continuous and therefore measurable, establishing (i). 
Moreover, since $\Theta$ is compact and $f$ is continuous, the image
$f(\Theta)$
is compact in
$\mathcal F(X)$.
The continuity of
$\mathcal D_{\sigma,\varepsilon}^2$
therefore implies that
$\Phi$
is uniformly bounded on
$\Theta$,
yielding (iii). 
These properties guarantee that the normalization constant
\[
Z :=
\int_{\Theta}
\exp\!\left(
-\lambda\,\Phi(\boldsymbol{\theta};g)
\right)
\,\Pi_0(d\boldsymbol{\theta})
\]
is strictly positive and finite, so that $\Pi$ defines a probability measure on $\Theta$ with Radon--Nikodým derivative
\[
\frac{d\Pi}{d\Pi_0}(\boldsymbol{\theta})
=
\frac{
\exp\!\left(
-\lambda\,\Phi(\boldsymbol{\theta};g)
\right)
}{Z}.
\]
Since $\Pi_0$ is a Radon measure on the compact metric space $\Theta$, $\Pi$ is also a Radon probability measure.
\end{proof}

This result follows closely the framework established for Bayesian inverse problems in Theorem 4.3 of \cite{Sullivan:2017}, adapted here to the symmetrized Sinkhorn-Gibbs posterior.

\subsection{Robustness of Symmetrized Sinkhorn-Gibbs Posteriors}
\label{sec:robustness}

Robustness concerns the stability of the posterior distribution with respect to perturbations in the observed data and the forward model, a critical property for reliable uncertainty quantification in inverse problems. While the well-definedness of the symmetrized Sinkhorn-Gibbs posterior follows from standard arguments in Bayesian inverse problems, establishing robustness requires additional regularity properties of the underlying loss function. 
Unlike posterior constructions based on genuine metrics—such as the Wasserstein distance employed in \cite{Motamed_Appelo:19,DunlopYang2022}—the Sinkhorn divergence does not satisfy a triangle inequality, preventing a direct application of classical stability arguments. Instead, we exploit the smoothness and $\dot H^{-1}$-regularity properties established in Section~\ref{sec:convexity}. We first establish a Lipschitz stability result for the normalized Sinkhorn divergence itself and then extend it to the symmetrized Sinkhorn risk introduced in Section~\ref{sec:stat_interp}. These estimates provide the key ingredients for proving stability of the resulting Gibbs posterior under perturbations of both the observed data and the forward model.

\begin{proposition}[$\dot H^{-1}$-Lipschitz Stability of the Normalized Sinkhorn Divergence]
\label{prop:Lipschitz_stability}
Let $f \in \mathcal{F}(X)$ be fixed, and let $g,g' \in \mathcal{F}(X)$ be mean-zero functions. Assume that $f$, $g$, and $g'$ belong to an $L^\infty$-bounded subset of $\mathcal{F}(X)$ and that the normalized densities $T_\sigma f$, $T_\sigma g$, and $T_\sigma g'$ are bounded above and below by positive constants. Then there exists a constant $C_\varepsilon>0$, depending on $\varepsilon$, $X$, $\sigma$, and the corresponding bounds, such that
\[
\left|
S_{\sigma,\varepsilon}(f,g)
-
S_{\sigma,\varepsilon}(f,g')
\right|
\le
C_\varepsilon
\|g-g'\|_{\dot H^{-1}(X)}.
\]
\end{proposition}

\begin{proof}
Set $\rho := T_\sigma f$, $\nu := T_\sigma g$, and $\nu' := T_\sigma g'$. Under the stated assumptions, these densities belong to a bounded subset of strictly positive probability densities supported on the compact domain $X$. 
For fixed $\rho$, the map $\nu \mapsto S_\varepsilon(\rho,\nu)$ is smooth on this class of densities. This follows from the smooth dependence of the entropic optimal transport potentials on the input measures \cite[Theorem~1]{Feydy_etal:19}. Hence, on bounded subsets, this map is locally Lipschitz with respect to the Wasserstein distance $W_2$, which metrizes weak convergence and thus provides a natural scale for Lipschitz-type estimates in the space of probability measures. In particular, there exists $L_\varepsilon>0$ such that
\[
\left|
S_\varepsilon(\rho,\nu)
-
S_\varepsilon(\rho,\nu')
\right|
\le
L_\varepsilon \,
W_2(\nu,\nu').
\]
For probability densities supported on a compact domain and bounded above and below by positive constants, the Wasserstein distance is controlled by the negative Sobolev norm \cite{Peyre:2018}. Thus,
\[
W_2(\nu,\nu')
\le
C_{X,M}
\|\nu-\nu'\|_{\dot H^{-1}(X)},
\]
where $C_{X,M}$ depends on the domain and the density bounds. 
Finally, since $\sigma$ is smooth and strictly increasing, the normalization map $T_\sigma$ is Lipschitz on bounded subsets of $\mathcal F(X)$ in the $\dot H^{-1}$ topology, using the same normalization argument as in Proposition~\ref{thm:Lipschitz}. Hence,
\[
\|T_\sigma g-T_\sigma g'\|_{\dot H^{-1}(X)}
\le
C_\sigma
\|g-g'\|_{\dot H^{-1}(X)}.
\]
Combining the three estimates gives
\[
\left|
S_{\sigma,\varepsilon}(f,g)
-
S_{\sigma,\varepsilon}(f,g')
\right|
\le
C_\varepsilon
\|g-g'\|_{\dot H^{-1}(X)}.
\]
\end{proof}

\begin{corollary}[$\dot H^{-1}$-Lipschitz Stability of the Symmetrized Sinkhorn Risk]
\label{cor:sym_risk_stability}
Under the assumptions of Proposition~\ref{prop:Lipschitz_stability}, there exists a constant $C_\varepsilon>0$ such that
\[
\left|
\mathcal D_{\sigma,\varepsilon}^{2}(f,g)
-
\mathcal D_{\sigma,\varepsilon}^{2}(f,g')
\right|
\le
C_\varepsilon
\|g-g'\|_{\dot H^{-1}(X)}.
\]
\end{corollary}

\begin{proof}
By definition,
\[
\mathcal D_{\sigma,\varepsilon}^{2}(f,g)
=
S_{\sigma,\varepsilon}^{2}(f,g)
+
S_{\sigma,\varepsilon}^{2}(-f,-g).
\]
Applying Proposition~\ref{prop:Lipschitz_stability} to the pairs $(f,g)$ and $(f,g')$, and similarly to $(-f,-g)$ and $(-f,-g')$, gives
\[
\left|
S_{\sigma,\varepsilon}(f,g)
-
S_{\sigma,\varepsilon}(f,g')
\right|
\le
C_1
\|g-g'\|_{\dot H^{-1}(X)},
\]
and
\[
\left|
S_{\sigma,\varepsilon}(-f,-g)
-
S_{\sigma,\varepsilon}(-f,-g')
\right|
\le
C_2
\|g-g'\|_{\dot H^{-1}(X)}.
\]
Using the identity $|a^2-b^2|=|a-b|\,|a+b|$ 
together with the boundedness of the Sinkhorn divergences on bounded subsets of $\mathcal F(X)$, yields
\[
\left|
S_{\sigma,\varepsilon}^{2}(f,g)
-
S_{\sigma,\varepsilon}^{2}(f,g')
\right|
\le
C_3
\|g-g'\|_{\dot H^{-1}(X)},
\]
and similarly
\[
\left|
S_{\sigma,\varepsilon}^{2}(-f,-g)
-
S_{\sigma,\varepsilon}^{2}(-f,-g')
\right|
\le
C_4
\|g-g'\|_{\dot H^{-1}(X)}.
\]
Summing the two inequalities gives the result.
\end{proof}

Having established the $\dot H^{-1}$-Lipschitz stability of the symmetrized Sinkhorn risk, we now quantify the impact of perturbations on the posterior itself. Following \cite{Stuart:10,Sullivan:2017}, we measure distances between posterior distributions using the Hellinger distance. Specifically, the Hellinger distance between two Borel probability measures $\Pi$ and $\Pi'$ on $\Theta$ is defined by
\[
d_{\mathrm H}(\Pi,\Pi')^2
=
\int_\Theta
\left(
\sqrt{\frac{d\Pi}{d\Pi_0}(\boldsymbol\theta)}
-
\sqrt{\frac{d\Pi'}{d\Pi_0}(\boldsymbol\theta)}
\right)^2
\Pi_0(d\boldsymbol\theta),
\]
where $\Pi_0$ denotes the prior measure. 
This framework allows us to establish robustness of the symmetrized Sinkhorn-Gibbs posterior with respect to perturbations in both the observed data and the forward model.

\begin{theorem}
\label{stability_data}
(Robustness to Observed Data Perturbations) 
Let $f:\Theta\rightarrow\mathcal F(X)$ be continuous, and suppose $\Pi_0$ is a Borel probability measure on $\Theta$. Let $\Pi$ and $\Pi'$ denote the symmetrized Sinkhorn-Gibbs posteriors corresponding to observed data $g$ and perturbed data $g'$, respectively. Then
\[
d_{\mathrm H}(\Pi,\Pi')
\le
C
\|g-g'\|_{\dot H^{-1}(X)},
\]
where the constant $C$ depends only on $\varepsilon$, $\lambda$, and uniform bounds on $f$, $g$, and $g'$.
\end{theorem}

\begin{proof}
By Corollary~\ref{cor:sym_risk_stability}, the risk functional satisfies
\[
|\Phi(\boldsymbol{\theta};g)-\Phi(\boldsymbol{\theta};g')|
\le
C_\varepsilon
\|g-g'\|_{\dot H^{-1}(X)}
\]
uniformly in $\boldsymbol{\theta}\in\Theta$. The result then follows from the standard Hellinger stability argument for Gibbs posteriors, which controls changes in the exponential weights and normalization constants in terms of uniform perturbations of the risk. For completeness, the full derivation is provided in Appendix~\ref{app:hellinger_stability}.
\end{proof}

\begin{theorem}
\label{stability_forward}
(Robustness to Forward Model Perturbations) 
Let $f:\Theta\rightarrow\mathcal F(X)$ be a continuous forward map, and let $f^h:\Theta\rightarrow\mathcal F(X)$ denote an approximation of $f$, where $h$ represents a discretization or approximation parameter. Suppose $\Pi_0$ is a Borel probability measure on $\Theta$, and let $g\in\mathcal F(X)$ be fixed observed data. Let $\Pi$ and $\Pi^h$ denote the corresponding symmetrized Sinkhorn-Gibbs posteriors generated by $f$ and $f^h$, respectively. Then
\[
d_{\mathrm H}(\Pi,\Pi^h)
\le
C
\max_{\boldsymbol\theta\in\Theta}
\|f(\boldsymbol\theta)-f^h(\boldsymbol\theta)\|_{\dot H^{-1}(X)},
\]
where the constant $C$ depends only on $\varepsilon$, $\lambda$, and uniform bounds on $f$, $f^h$, and $g$.
\end{theorem}

\begin{proof}
Define
\[
\Phi(\boldsymbol\theta)
=
\mathcal D_{\sigma,\varepsilon}^{2}
\bigl(
f(\boldsymbol\theta),
g
\bigr),
\qquad
\Phi^h(\boldsymbol\theta)
=
\mathcal D_{\sigma,\varepsilon}^{2}
\bigl(
f^h(\boldsymbol\theta),
g
\bigr).
\]
Applying Corollary~\ref{cor:sym_risk_stability} with the first argument perturbed yields
\[
|\Phi(\boldsymbol\theta)-\Phi^h(\boldsymbol\theta)|
\le
C_\varepsilon
\|f(\boldsymbol\theta)-f^h(\boldsymbol\theta)\|_{\dot H^{-1}(X)},
\]
uniformly in $\boldsymbol\theta$. 
Taking the maximum over $\Theta$ and repeating the standard Hellinger-distance argument used in Theorem~\ref{stability_data} gives the desired estimate.
\end{proof}

These results show that the symmetrized Sinkhorn-Gibbs posterior inherits robustness properties directly in the weaker $\dot H^{-1}$ topology. In contrast to posterior constructions whose stability is naturally controlled in stronger norms such as $L^2$ (cf.~\cite{Stuart:10}), the proposed framework remains stable under low-frequency perturbations of both the observed data and the forward model. Combined with the asymptotic convexity and stability properties of the symmetrized Sinkhorn loss established in Section~\ref{sec:convexity}, these robustness guarantees provide theoretical support for the use of symmetrized Sinkhorn-Gibbs inference in oscillatory inverse problems.


\section{Numerical Validation of Symmetrized Sinkhorn-Gibbs Inference}
\label{sec:numerical_exp}

In this section, we assess the performance of the proposed symmetrized Sinkhorn-Gibbs inference framework on a simplified yet structurally challenging seismic inversion problem. We compare its behavior with Gibbs posteriors constructed from two alternative empirical risks: the Euclidean ($L^2$) risk and a trace-by-trace Wasserstein risk. The goal is twofold. First, we investigate how different loss functions shape the empirical risk landscape of the inverse problem, especially its oscillations, local minima, and basin structure near the true parameter. Second, we compare the resulting posterior distributions across multiple regimes, including varying data resolution, observational noise, and intrinsic variability in the underlying parameters.

Unlike classical inference benchmarks, which typically assume a fixed but unknown parameter with noisy data and focus on posterior concentration as the amount of data increases, we consider a benchmark with both \emph{epistemic} and \emph{aleatoric} uncertainty in single-event and multi-event settings, respectively. At the event level, the problem reduces to a standard inverse problem with a fixed parameter, allowing us to examine posterior concentration and sensitivity to resolution and noise. At the population level, however, parameters are modeled as random variables drawn from an underlying distribution, enabling us to assess how well the inferred posteriors recover this distribution from a finite number of observations. This dual structure provides a richer validation setting: it allows us to study both classical consistency behavior and finite-sample distributional recovery, and to compare inferred posteriors directly against a known ground-truth distribution.

\subsection{Benchmark and experimental framework}
\label{sec:benchmark}

We consider a simplified seismic inversion problem in which the goal is to infer the location of a source along a one-dimensional geological fault, and potentially other seismic parameters, from observed wavefield data. The feature domain is taken as a space-time domain
\[
X = [0,1] \times [0,T],
\]
where $x \in [0,1]$ denotes the spatial coordinate along the fault and $t \in [0,T]$ denotes time. The wavefield is observed on a discrete grid of $N_x$ receivers and $N_t$ time samples.

\paragraph{Forward model.}
We denote by $f(\theta)$ the forward model, where the parameter $\theta = x_s \in [0,1]$ represents the source location. The predicted wavefield is constructed as a superposition of multiple arrivals:
\[
f(x,t; x_s) = a \sum_{\ell=1}^4 A_\ell(x; x_s)\, s_\ell\!\left(t - T_\ell(x; x_s)\right).
\]
Each term corresponds to a distinct propagation path, characterized by a travel-time map $T_\ell(x; x_s)$, a spatial amplitude modulation $A_\ell(x; x_s)$, and a source-time waveform $s_\ell$. The waveforms are defined as mixtures of shifted Ricker wavelets,
\[
s_\ell(\tau) = \sum_{m=1}^3 \alpha_{\ell,m}\, \psi_{\nu_\ell}(\tau - \tau_{\ell,m}),
\]
where $\psi_{\nu}$ denotes a Ricker wavelet with dominant frequency $\nu$:
\[
\psi_{\nu}(t) = \big(1 - 2\pi^2 \nu^2 t^2\big)e^{-\pi^2 \nu^2 t^2}.
\]

The four arrivals include direct, delayed, reflected, and scattered components, with distinct travel-time geometries and spatial amplitude modulations:
\begin{enumerate}
\item A direct arrival with linear travel time
\[
T_1(x; x_s) = \frac{|x - x_s|}{c},
\qquad
A_1(x; x_s) = r_{\mathrm{dir}}.
\]
\item A delayed direct-path arrival
\[
T_2(x; x_s) = \Delta_0 + \frac{|x - x_s|}{c},
\qquad
A_2(x; x_s) = r_0.
\]
\item A reflected arrival with curved travel-time structure and spatially oscillatory branch
\[
T_3(x; x_s) = \Delta_1 + \beta_1 (x - \xi_1)^2 + \gamma_1 |x_s - \xi_1|,
\]
\[
A_3(x; x_s) = r_1 \left[1 + \rho_1 \cos\!\big(k_1 (x - x_s - \eta_1)\big)\right].
\]
\item A scattered/partially visible arrival with oscillatory and spatially localized branch
\[
T_4(x; x_s) = \Delta_2 + \beta_2 \sqrt{(x - \xi_2)^2 + h_2^2} + \gamma_2 |x_s - \xi_2|,
\]
\[
A_4(x; x_s) =
r_2 \, w(x)\left[1 + \rho_2 \cos\!\big(k_2 (x - x_s - \eta_2)\big)\right],
\]
where $w(x)$ is a smooth visibility window supported on $[0,1]$.
\end{enumerate}

\paragraph{Structure of the wavefield.} 
This construction produces wavefields with multiple overlapping arrivals, spatially varying amplitudes, and oscillatory patterns in both time and space. These features mimic key characteristics of seismic data, including wavefront interference, path-dependent waveform distortion, and partial visibility of certain arrivals. While simplified and low-dimensional, the model isolates fundamental mechanisms—misalignment, interference, and limited visibility—that are known to create challenging inverse problems. The benchmark is specifically designed to induce misalignment in both time and space, giving rise to the temporal and spatial cycle-skipping phenomena discussed in Section~\ref{sec:cycle_skipping}. As a result, standard discrepancy measures often exhibit highly nonconvex loss landscapes. The purpose of this construction is not to reproduce a particular physical system in full detail, but rather to provide a controlled and reproducible benchmark that isolates key sources of nonconvexity and enables systematic comparison of inference methodologies, including the proposed symmetrized Sinkhorn-Gibbs framework.

\paragraph{Data generation.}
We consider a setting that models intrinsic variability in the system, corresponding to the spatial distribution of earthquake source locations along a geological fault. Each seismic event represents a distinct realization of the system, with its own parameter $\theta_j \in \Theta$ drawn from an underlying probability measure $\mu_\theta$ on $\Theta$.

We consider a collection of $J$ independent events. For each event $j = 1,\dots,J$, the parameter is drawn independently as
\[
\theta_j \overset{\mathrm{iid}}{\sim} \mu_\theta,
\]
and the corresponding data field is generated through the forward model:
\[
g_j(\mathbf{x}) = f(\theta_j)(\mathbf{x}), \qquad \mathbf{x} = (x,t) \in X.
\]

In practice, each field is observed on a discrete grid $\{\mathbf{x}_i\}_{i=1}^n \subset X$, yielding observations $\{g_j(\mathbf{x}_i)\}_{i=1}^n$, which we denote compactly by $g_j$. Here, $n = N_x \times N_t$ corresponds to the number of spatial and temporal samples.

To assess robustness, we consider both noise-free and noisy observations. In the noisy case,
\[
g_j^{\mathrm{obs}} = g_j + \varepsilon_j,
\]
where $\varepsilon_j$ represents additive observational noise.

The full dataset consists of $J$ independent realizations:
\[
\mathcal{D}_J = \left\{ \{ g_j^{\mathrm{obs}}(\mathbf{x}_i) \}_{i=1}^n \right\}_{j=1}^J.
\]

To illustrate the structure of the benchmark data, Figure~\ref{fig:benchmark_topviews} shows three representative noise-free events corresponding to different source locations, namely $x_s = 0.2$, $0.5$, and $0.7$. Each panel displays a top view of the observed wavefield as a function of time $t$ and spatial location $x$. The plots highlight the main qualitative features built into the benchmark: multiple arrivals, curved travel-time branches, oscillatory spatial amplitude patterns, and partial visibility of some arrivals. As the source location varies, the relative position and interaction of these structures change in a coherent but nontrivial manner, illustrating the geometric complexity that makes the benchmark challenging for inference.
\begin{figure}[!ht]
\centering
\includegraphics[width=0.31\textwidth]{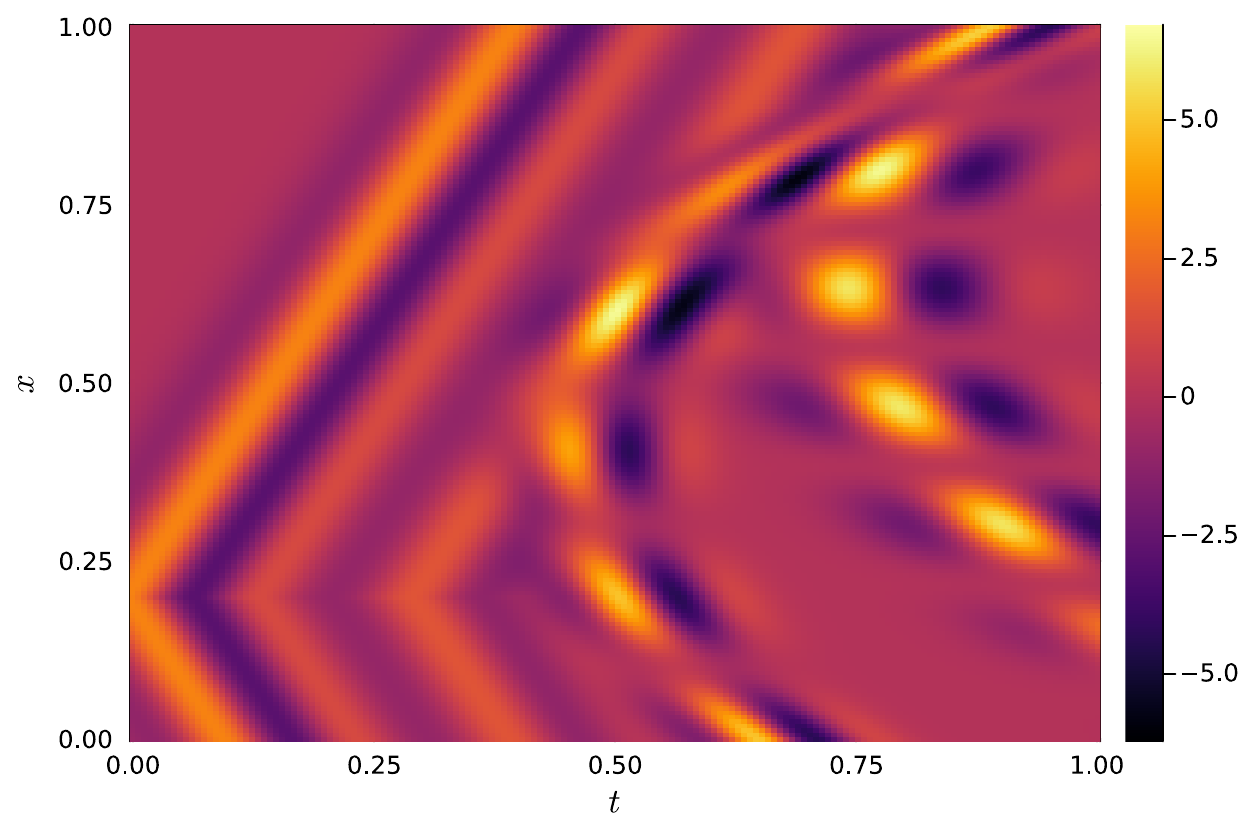}
\includegraphics[width=0.31\textwidth]{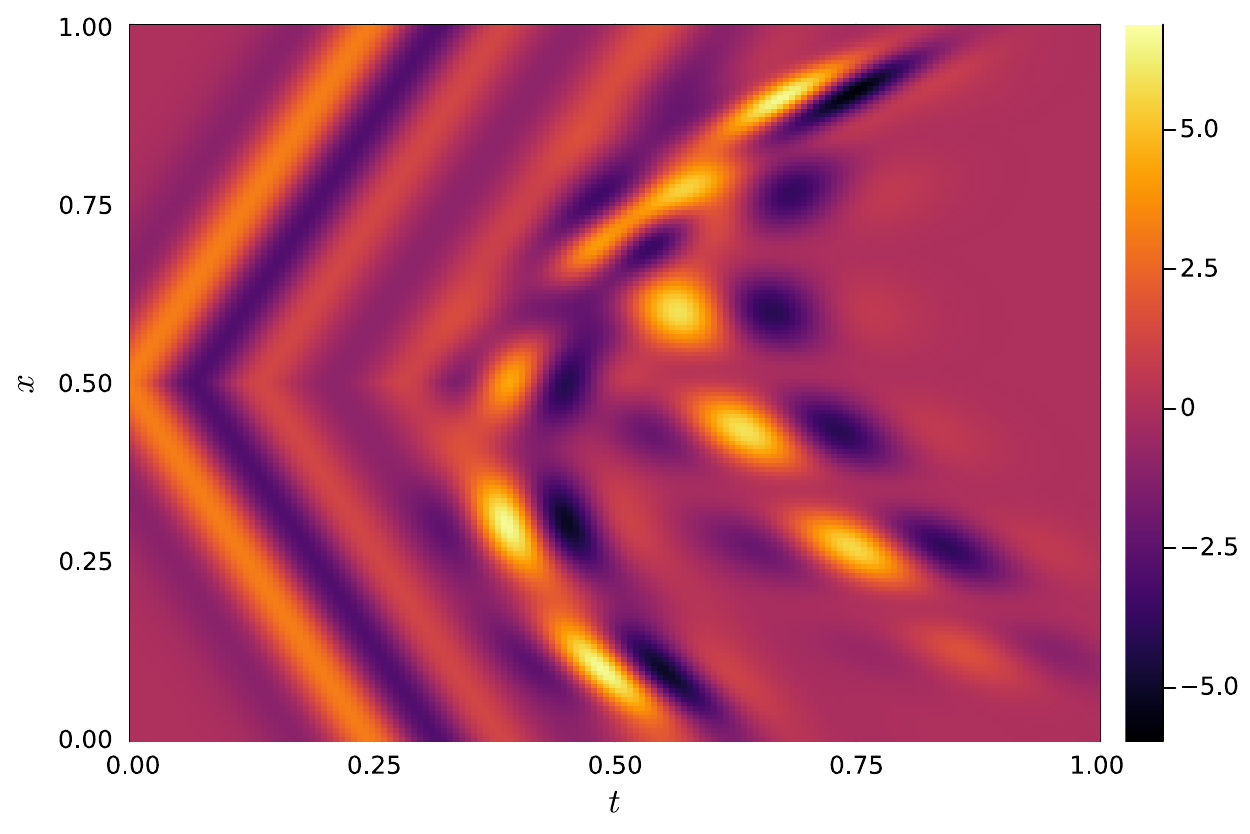}
\includegraphics[width=0.31\textwidth]{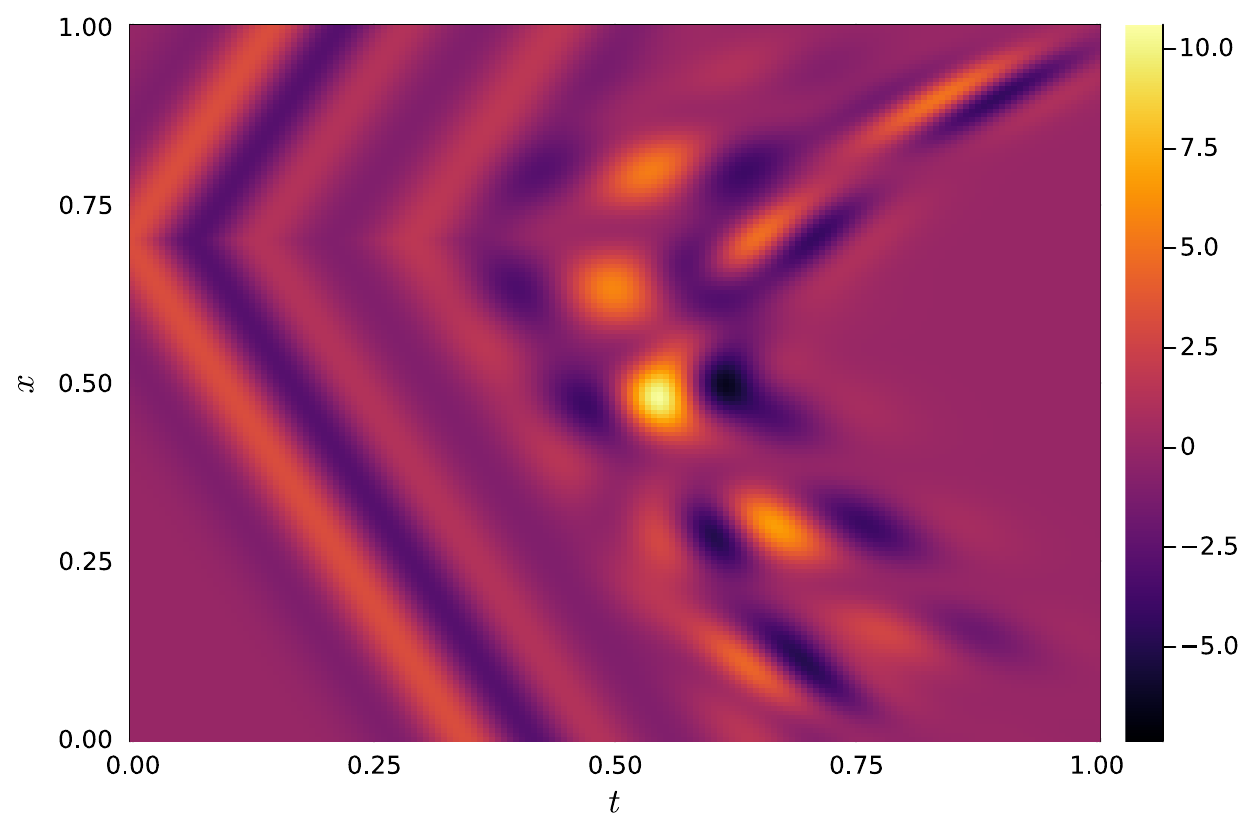}
\caption{Three representative benchmark events shown as top views of the noise-free observed wavefield over the space-time domain $(x,t)$, for source locations $x_s = 0.2$, $0.5$, and $0.7$ (left to right). The wavefields exhibit multiple arrivals, oscillatory spatial structure, and partial visibility, all of which contribute to the nonconvexity mechanisms studied later in the section.}
\label{fig:benchmark_topviews}
\end{figure}

\paragraph{Epistemic and aleatoric uncertainty.}
We consider two complementary regimes.
\begin{itemize}
    \item \emph{Epistemic uncertainty:} A fixed parameter $\theta^\ast$ (which is $\theta_j$ at event $j$) generates the data $g_j$ (or $g_j^{\mathrm{obs}}$), and the goal is to infer this parameter from a single dataset. This corresponds to the classical inverse problem setting.

    \item \emph{Aleatoric uncertainty:} The parameters $\{\theta_j\}_{j=1}^J$ are modeled as independent samples from an unknown distribution $\mu_\theta$. In this case, each dataset corresponds to a different realization of the system, and the goal is to infer properties of the underlying distribution.
\end{itemize}
This combined setting allows us to study both posterior concentration at the event level and distributional recovery at the population level.

\paragraph{Inference framework.}
For each event $j$, we infer the parameter $\theta_j$ from the observations $g_j$ using a Gibbs posterior of the form
\[
\Pi_j(d\theta \mid g_j) \propto \exp\!\bigl(-\lambda n \, \Phi(\theta; g_j)\bigr)\, \Pi_0(d\theta),
\]
where $\Phi(\theta; g_j)$ denotes the empirical risk defined in Section~\ref{sec:background}, with the loss function chosen as either the $L^2$ loss, a trace-by-trace Wasserstein loss, or the proposed Sinkhorn-based loss.

\paragraph{Two levels of evaluation.}
In this combined epistemic-aleatoric setting, we distinguish two levels of evaluation.

\smallskip
\noindent
\emph{Event-level inference.} This relates to the epistemic uncertainty at the level of individual events. For each event $j$, we analyze posterior concentration, estimation error, and sensitivity to resolution, initialization, and noise. This corresponds to the classical setting of recovering a fixed parameter. As the information content of the data increases, we expect the posterior to concentrate near the true parameter $\theta_j$:
\[
\Pi_j(\cdot \mid g_j) \;\to\; \delta_{\theta_j},
\]
where $\delta_{\theta_j}$ denotes the Dirac measure at $\theta_j$.

\smallskip
\noindent
\emph{Population-level inference.} This relates to the aleatoric uncertainty component. Across multiple events, we evaluate the recovery of the underlying distribution $\mu_\theta$, which we interpret as a probability measure on $\Theta$. From $J$ independent datasets, corresponding to parameters $\{\theta_j\}_{j=1}^J$ sampled from $\mu_\theta$, we construct:
\[
\hat{\mu}_J^{\mathrm{post}} 
= \frac{1}{J} \sum_{j=1}^J \Pi_j(d\theta \mid g_j),
\]
and compare this estimator with the true distribution $\mu_\theta$ using the quadratic Wasserstein distance.

\paragraph{Resolution and repetition.}
To study convergence behavior, we vary the data resolution by increasing $(N_x, N_t)$, thereby increasing the information content of each event. For population-level experiments, we vary the number of events $J$ and repeat experiments multiple times to obtain statistically stable estimates. We note that for each repetition, a new 
set of $J$ parameters is drawn from the target $\mu_\theta$, while all other components (forward model, sampling grid, amplitudes, and inference 
settings) are kept fixed. This setup allows us to examine both asymptotic trends (in resolution or number of events) and finite-sample behavior, providing a comprehensive evaluation of the proposed inference framework.

\paragraph{Organization of the numerical study.} 
The numerical study proceeds from increasingly demanding inference tasks. We first analyze the geometry of the underlying loss functions and the resulting posterior distributions in a one-parameter setting. We then examine robustness under observational noise and assess the impact of event-level inference on population-level recovery. In these experiments, inverse-temperature parameters are calibrated separately for each method to ensure comparable posterior concentration.  Finally, we consider a multi-parameter inverse problem in which model parameters and inverse temperature are inferred jointly using the adaptive Gibbs sampler, where the inverse temperature is learned jointly with the physical parameters.

\subsection{Mechanisms of nonconvexity: temporal and spatial cycle skipping}
\label{sec:cycle_skipping}

As discussed in Section~\ref{sec:background}, oscillatory inverse problems often exhibit highly nonconvex discrepancy landscapes arising from signal misalignment. The benchmark introduced above is designed to isolate and amplify these effects through structured temporal and spatial cycle skipping.

\paragraph{Temporal cycle skipping.}
Each waveform $s_\ell(t)$ contains oscillatory structure in time. When the parameter $\theta$ is misspecified, the predicted arrivals are shifted relative to the observed signal. Due to the oscillatory nature of the waveforms, incorrect parameter values may align peaks of the predicted signal with different oscillation cycles of the observed signal. This produces multiple local minima in the $L^2$ loss, even when the misalignment is large.

\paragraph{Spatial cycle skipping.}
In addition to temporal structure, the amplitude modulations $A_\ell(x;\theta)$ introduce oscillatory patterns across receivers. As the parameter varies, these spatial patterns shift along the domain. Because of their periodic structure, incorrect parameter values may align with different spatial phases of the observed data, again leading to multiple local minima. This effect is particularly pronounced when interference between arrivals produces complex spatial patterns.

\paragraph{Implications for loss functions.}
These two mechanisms affect loss functions differently:
\begin{itemize}
\item The $L^2$ loss is sensitive to pointwise differences and therefore suffers from both temporal and spatial cycle skipping.
\item Trace-by-trace Wasserstein distances mitigate temporal misalignment by aligning signals independently in time at each receiver, but they do not account for spatial coupling and therefore remain sensitive to spatial cycle skipping.
\item In contrast, optimal transport in the joint space-time domain allows mass to move simultaneously in both space and time. The resulting Sinkhorn-based loss can therefore accommodate both temporal and spatial misalignment, leading to an empirical risk landscape with reduced oscillations, fewer competing local minima, and a broader basin near the true parameter.
\end{itemize}

\medskip

This construction is intentional: the benchmark isolates the fundamental mechanisms that lead to nonconvexity in practice, while remaining sufficiently controlled to allow systematic comparison across loss functions. In this sense, it serves not only as a test problem for the present work, but also as a structured validation setting for transport-based inference methods more broadly.

\subsection{Single-event analysis: risk landscape}
\label{sec:loss_landscape}

We begin with a single-event setting in which the parameter $\theta = x_s$, representing the source location of an earthquake, is fixed. Before performing inference, we examine the 
empirical risk landscape as a function of the source-location parameter. 
We compare three empirical risks within the Gibbsian framework introduced in Section~\ref{sec:adaptive_mcmc}: the Euclidean ($L^2$) risk, a trace-by-trace Wasserstein risk (defined as the average over receivers of squared one-dimensional Wasserstein distances in time), and the proposed symmetrized Sinkhorn risk.

For a given dataset $g$, generated from a single realization of $x_s$, we compute the three empirical risks over a range of candidate source locations $x_s \in \Theta = [0,1]$. We refer to the resulting function $\Phi(\theta; g)$ as the \emph{risk landscape}. Figure~\ref{fig:risk_landscapes} shows these landscapes for increasing data resolutions.

\begin{figure}[!ht]
\centering
\includegraphics[width=0.31\textwidth]{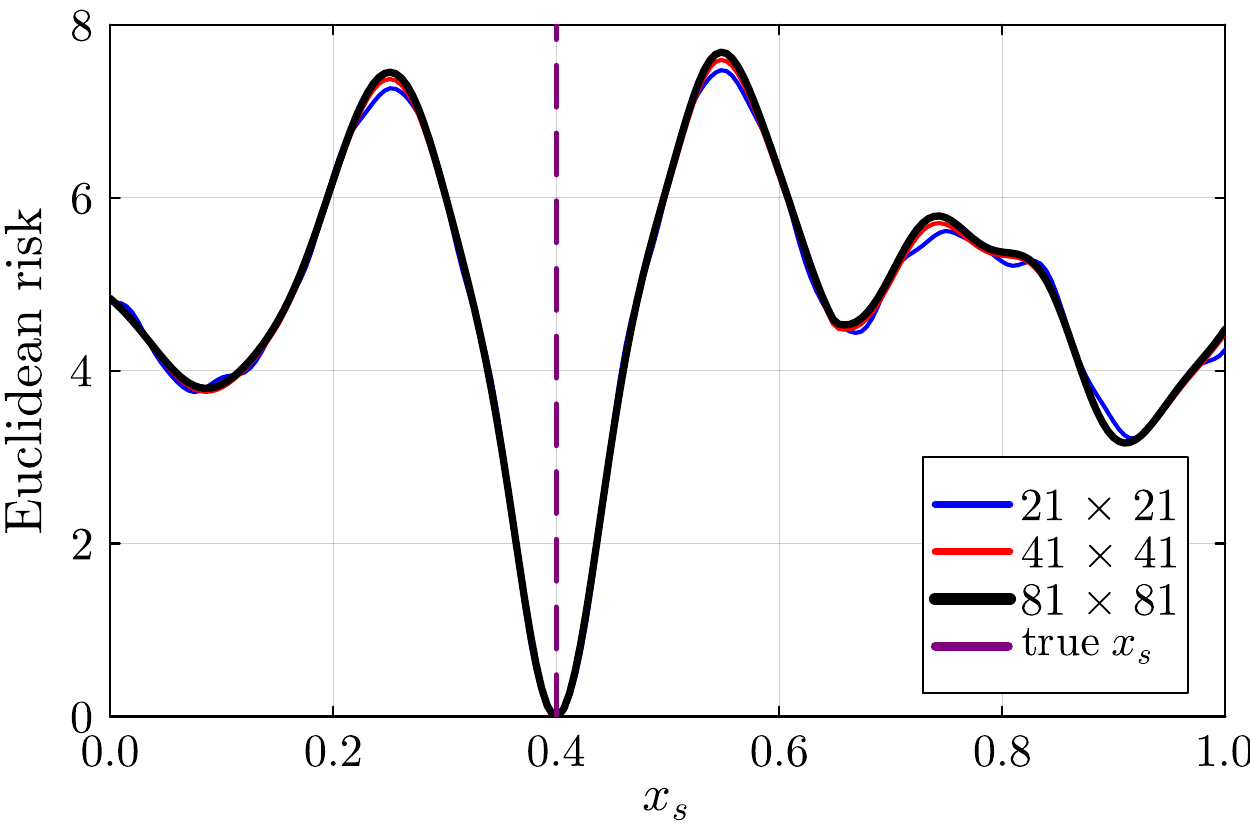}
\includegraphics[width=0.31\textwidth]{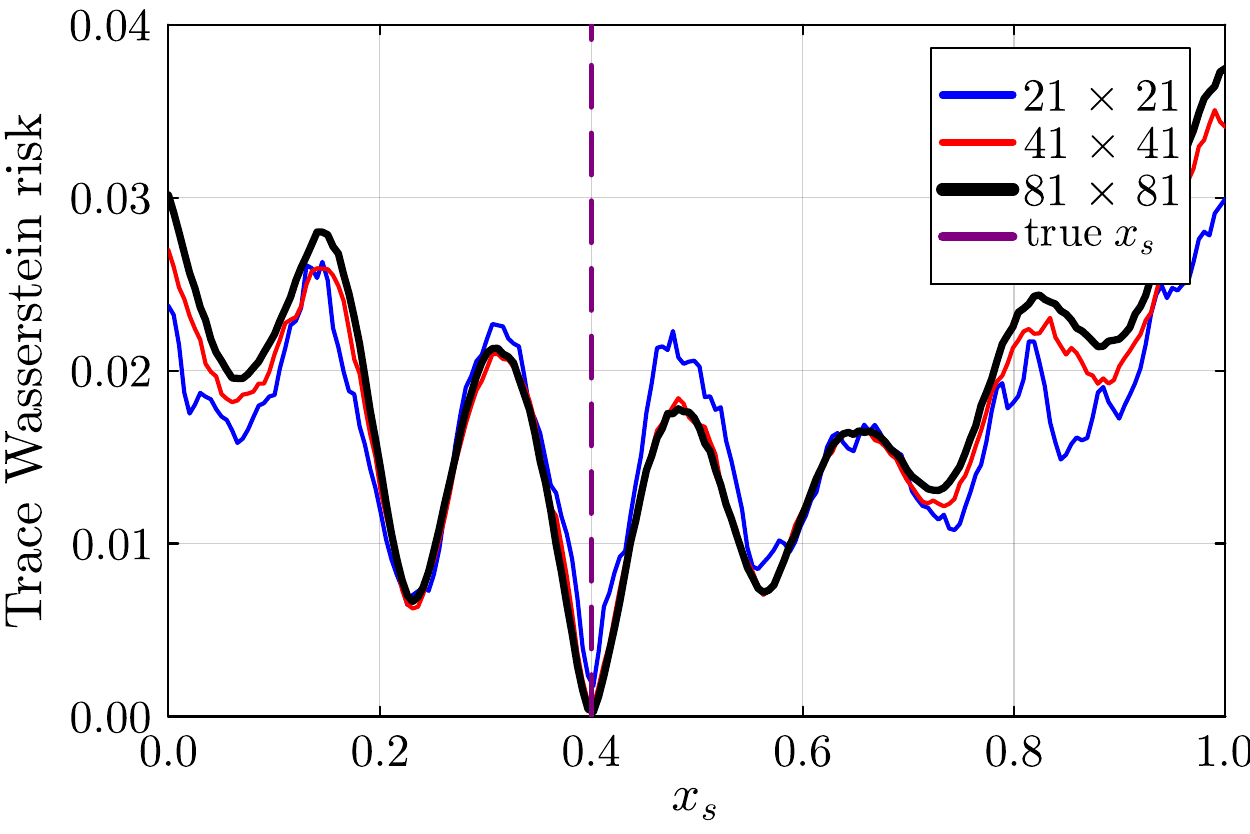}
\includegraphics[width=0.31\textwidth]{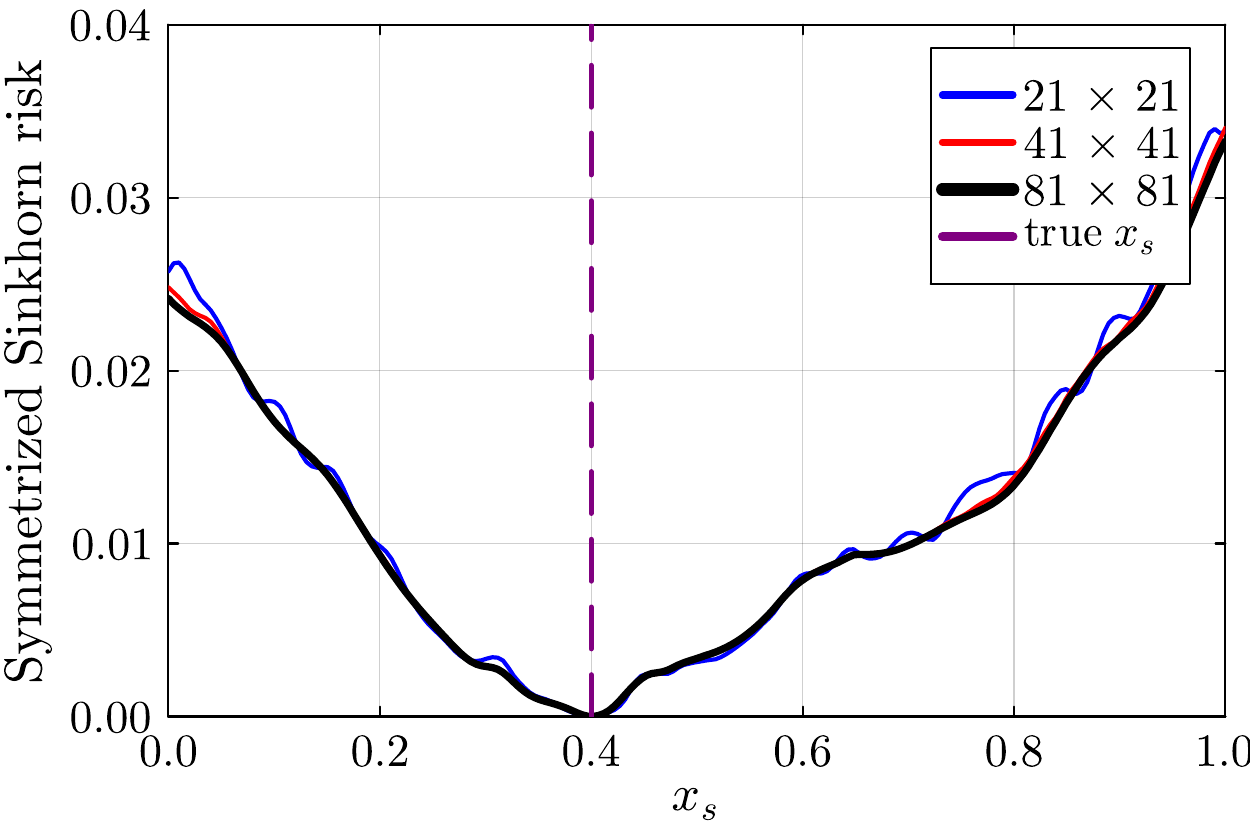}
\caption{Empirical risk landscapes as functions of the source location $x_s$ for increasing data resolutions. From left to right: Euclidean, trace Wasserstein, and symmetrized Sinkhorn risks. The curves correspond to progressively refined space-time grids, with the finest resolution shown in black. The true source location ($x_s=0.4$) is indicated by the dashed vertical line.}
\label{fig:risk_landscapes}
\end{figure}

All three risks are minimized near the true parameter in this noise-free setting. However, their landscape structure differs substantially. 
The Euclidean risk is strongly nonconvex, exhibiting pronounced oscillations and multiple secondary extrema. These arise from cycle-skipping: pointwise comparison of oscillatory signals admits multiple competing alignments, leading to a highly irregular landscape. 
The trace Wasserstein risk reduces temporal misalignment by aligning signals independently at each receiver. This attenuates large-scale oscillations compared to the Euclidean case, but the landscape remains visibly nonconvex, with persistent local minima across resolutions. The residual oscillations reflect the fact that transport is performed only in time, while spatial structure is aggregated through averaging across receivers. 
In contrast, the symmetrized Sinkhorn risk is smoother and more stable across resolutions. The curves largely align as the resolution is refined and exhibit a broad basin around the true parameter. By performing transport in the joint space-time domain, this risk reduces both temporal and spatial misalignment, resulting in a more regular landscape in the region relevant for inference.

Overall, these observations indicate that the Sinkhorn-based risk leads to a risk landscape that is less oscillatory and less dominated by competing local minima, even in this challenging oscillatory setting.

\subsection{Single-event inference: posterior behavior}
\label{sec:posterior_behavior}

We now perform inference in the single-event setting, where the parameter $\theta = x_s$ is fixed and the goal is to recover it from a single dataset $g$. 
This corresponds to the classical inverse problem setting and allows us to examine how differences in the empirical risk landscape translate into posterior concentration, accuracy, and sensitivity to spurious modes.

Throughout this subsection, we compare the Gibbs posteriors induced by the Euclidean, Trace Wasserstein, and Symmetrized Sinkhorn risks. To reduce dependence on a single representative event, all reported quantities are averaged over the set of true source locations
\[
x_s^\ast \in \{0.1,\,0.2,\,0.3,\,0.4,\,0.5,\,0.6,\,0.7,\,0.8,\,0.9\}.
\]
For all experiments, the MCMC chain is initialized at the common value $x_s^{(0)}=0.5$, corresponding to the midpoint of the uniform prior on $[0,1]$, and each experiment is repeated over 50 random seeds in order to assess sensitivity to the stochastic variability of the sampling procedure.

In this one-parameter setting, inverse temperatures are calibrated separately for each risk in order to facilitate a controlled comparison of posterior geometry. The adaptive temperature-learning strategy introduced in Section~\ref{sec:adaptive_mcmc} is reserved for the multi-parameter experiments considered later in the section. The inverse temperatures are chosen so that the resulting posteriors have comparable variance while maintaining small posterior mean errors. This allows the comparison to focus on differences in the geometry of the underlying risks rather than on differences induced by temperature tuning. The parameter $\delta$ used in the transport-based risks is also fixed after calibration. Values that are too small or too large were observed to degrade performance in different ways, so an intermediate value is used throughout.

\subsubsection{Effect of data resolution (noise-free case)}

We first study how the posterior changes as the data resolution increases in the noise-free setting. For each resolution level, we compute two summary quantities: the average absolute error of the posterior mean, $|\hat{x}_s-x_s^\ast|$, and the average posterior variance. The first measures inference accuracy, while the second measures posterior concentration.

We consider three levels of grid refinement corresponding to the resolutions $21\times21$ (coarse), $41\times41$ (medium), and $81\times81$ (fine). For each resolution, the quantities are first averaged over multiple choices of the true source location $x_s^\ast$. We then summarize the seed-to-seed variability using the median across 50 random seeds, with shaded bands corresponding to the interquartile range, i.e., the 25\%--75\% quantiles. Figure~\ref{fig:posterior_resolution} displays these results as functions of the data resolution.

\begin{figure}[!ht]
\centering
\includegraphics[width=0.48\textwidth]{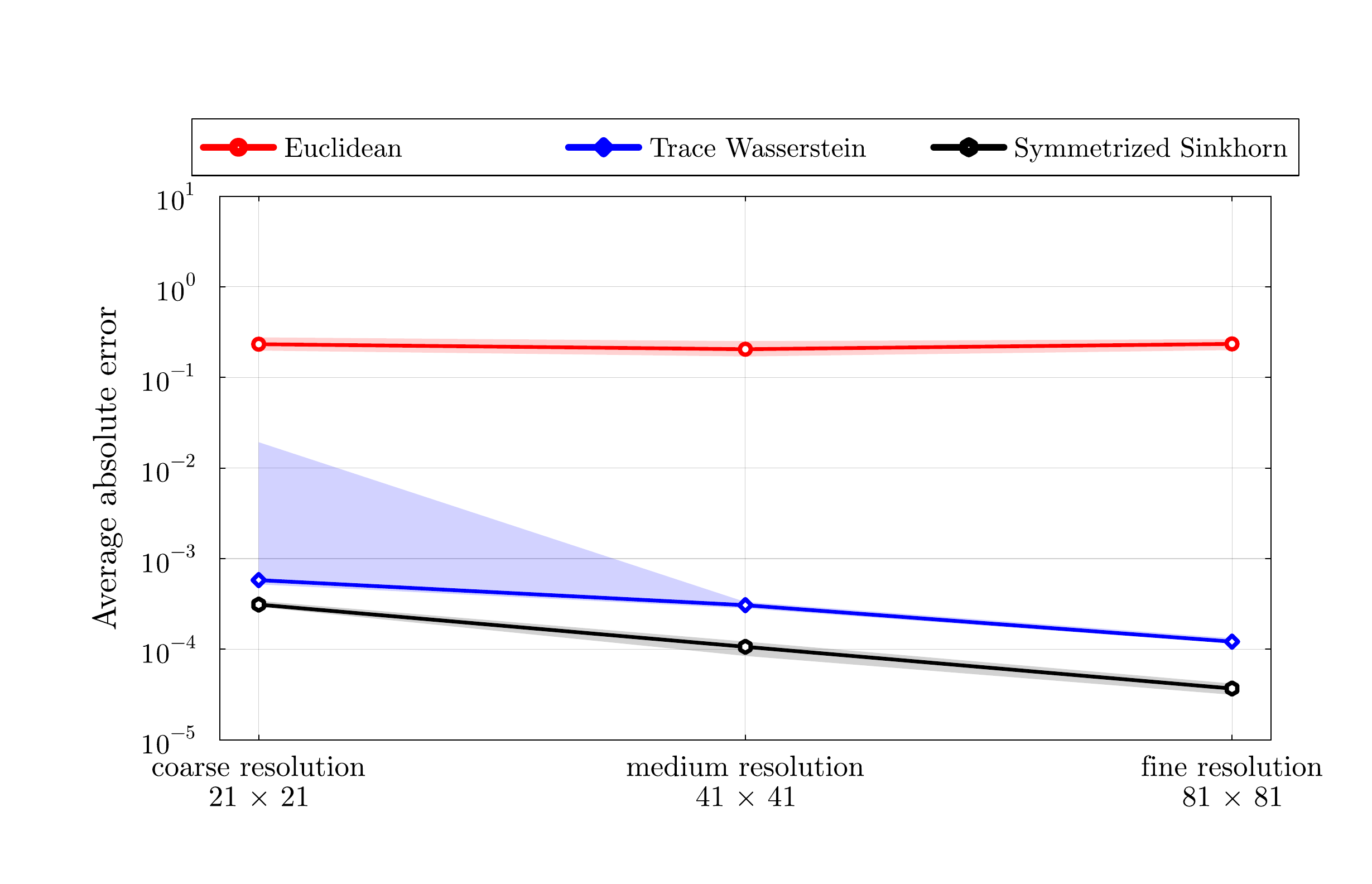}
\includegraphics[width=0.48\textwidth]{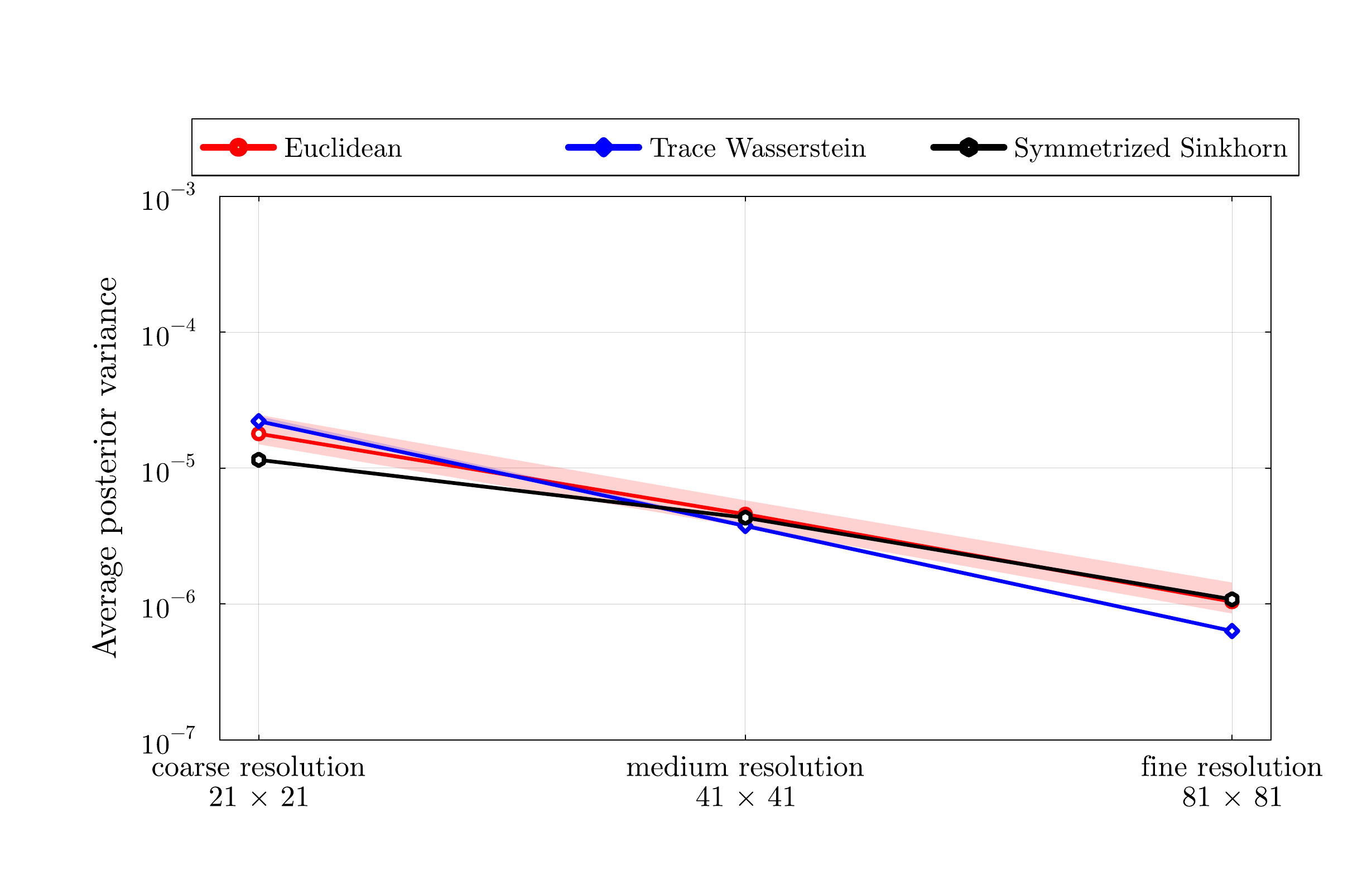}
\caption{Posterior behavior as a function of data resolution in the noise-free single-event setting. Left: average absolute error of the posterior mean. Right: average posterior variance. In each panel, values are first averaged over multiple choices of the true source location $x_s^\ast$ and then summarized over 50 random seeds. Solid curves show the median across seeds, while shaded regions indicate the interquartile range (25\%--75\% quantiles).}
\label{fig:posterior_resolution}
\end{figure}

As the resolution increases, all three methods exhibit decreasing posterior variance, indicating that finer data provide more information. However, the corresponding behavior of the posterior mean error differs substantially across the three posteriors.

The Euclidean posterior remains inaccurate across all resolutions. Its posterior mean error stays large as the grid is refined, despite the fact that its posterior variance decreases at a rate comparable to the transport-based methods. This indicates that the Euclidean posterior becomes increasingly concentrated around incorrect modes induced by the highly nonconvex Euclidean risk landscape. In this sense, the Euclidean posterior is not only inaccurate but also overconfident.

The Trace Wasserstein posterior improves substantially over the Euclidean case, achieving errors several orders of magnitude smaller. However, its variability across seeds remains noticeably larger, particularly at coarse resolution, where the interquartile band is significantly wider. This behavior is consistent with the residual nonconvexity of the trace-wise Wasserstein risk landscape: temporal misalignment is reduced, but spatially induced cycle-skipping continues to produce competing alignments across receivers.

In contrast, the Symmetrized Sinkhorn posterior exhibits the most stable and coherent behavior. Its posterior mean error decreases systematically with resolution and remains uniformly smallest across all grids, while the corresponding interquartile bands remain narrow. At the same time, its posterior variance is comparable to that of the competing methods. Thus, its improved accuracy is not simply a consequence of a more diffuse posterior, but rather of a less oscillatory and less multimodal structure of the underlying risk landscape. The joint space-time Sinkhorn comparison mitigates both temporal and spatial misalignment, leading to a posterior that is simultaneously accurate, stable, and well calibrated.

These results demonstrate that improvements in risk geometry translate directly into improvements in posterior behavior. In the noise-free single-event regime, the Symmetrized Sinkhorn posterior achieves the most reliable combination of accuracy, stability across seeds, and uncertainty quantification.

\subsubsection{Robustness with respect to noise}

We next study the robustness of the three posteriors with respect to observational noise. In this experiment, we fix the coarse observation grid ($21\times 21$) and vary the relative noise level over a prescribed range. We consider relative noise levels
$$
\eta \in \{1\%,2\%,4\%,6\%,8\%,10\%\},
$$
where the percentage is defined relative to the maximum amplitude of the clean signal. For a given noise level $\eta$, the noisy dataset is generated according to
$$
g^\eta = g + \sigma \, \xi,
\qquad
\sigma = \eta \, \max |g|,
$$
where $\xi$ is a standard Gaussian random vector of the same dimension as the data.

For each true source location, each noise level, and each random seed, we compare the posterior obtained from noisy data with the posterior obtained from clean data. As a measure of posterior stability, we compute the Hellinger distance between the clean-data posterior and the noisy-data posterior. 
The use of the Hellinger distance is consistent with the theoretical robustness results of Section~\ref{sec:robustness}, allowing empirical and theoretical notions of posterior stability to be compared directly. We also compute the absolute error of the posterior mean under noisy data. The quantities are first averaged over the set of true source locations. We then summarize the resulting seed-to-seed variability using the median across 50 random seeds, with the shaded bands corresponding to the interquartile range, i.e., the 25\%--75\% quantiles. Figure~\ref{fig:posterior_noise_sweep} summarizes these results as functions of the noise level. 
\begin{figure}[!ht]
\centering
\includegraphics[width=0.48\textwidth]{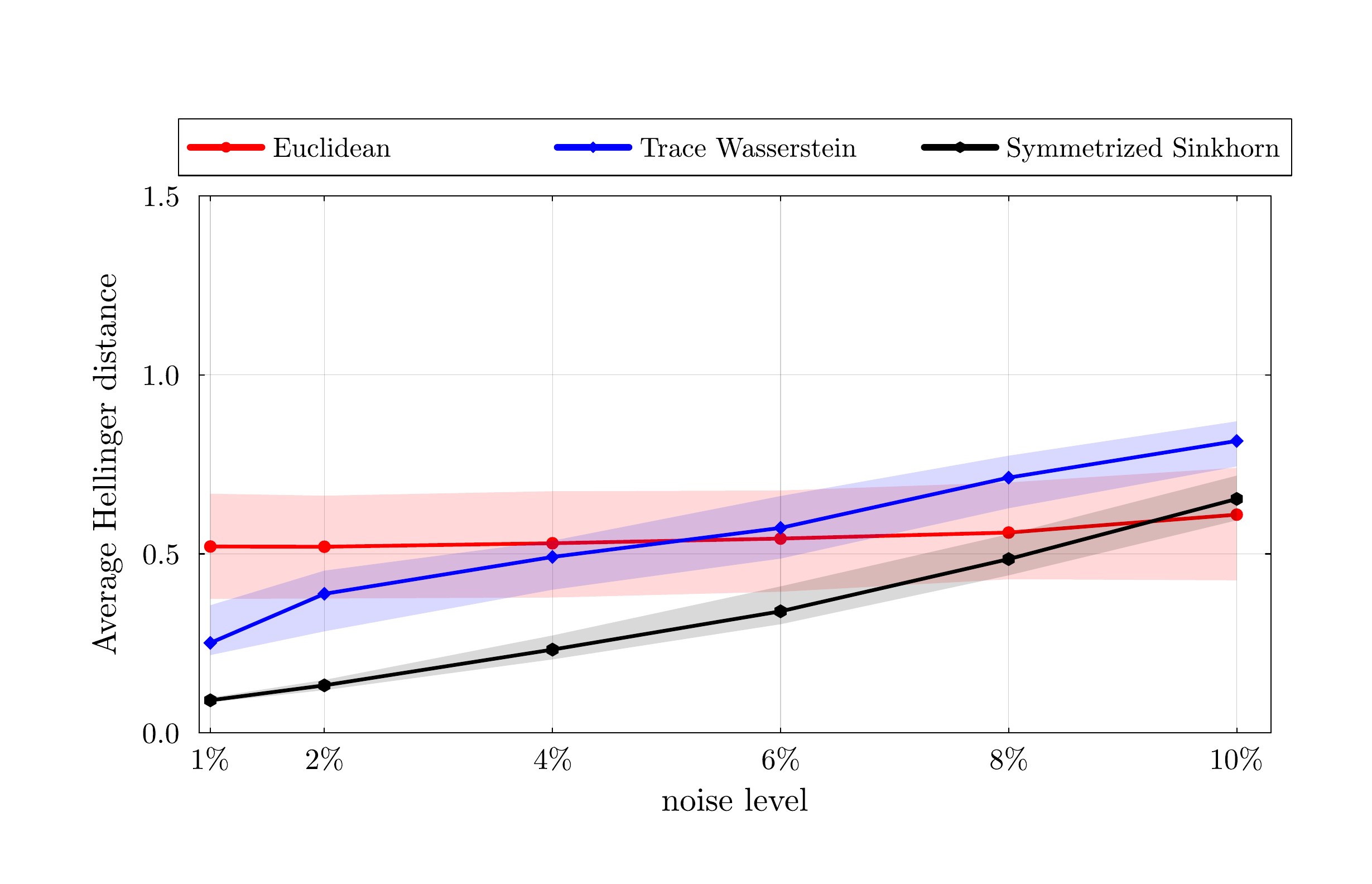}
\includegraphics[width=0.48\textwidth]{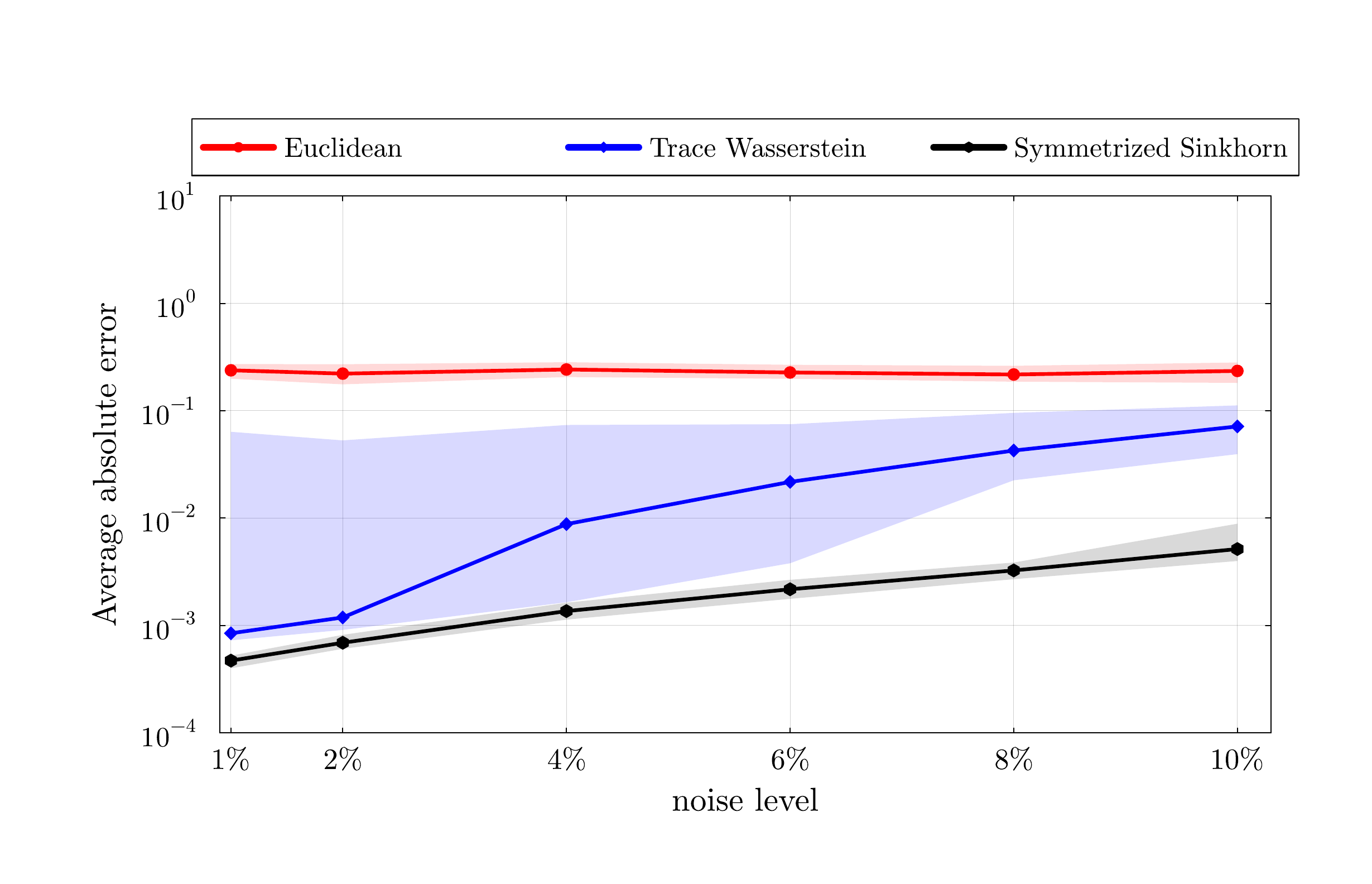}
\caption{Noise robustness in the single-event setting on the coarse observation grid ($21\times 21$). Left: average Hellinger distance between the clean-data and noisy-data posteriors. Right: average absolute error of the posterior mean under noisy data. In each panel, values are first averaged over multiple choices of the true source location $x_s^\ast$ and then summarized over 50 random seeds. Solid curves show the median across seeds, while shaded regions indicate the interquartile range (25\%--75\% quantiles).}
\label{fig:posterior_noise_sweep}
\end{figure}

Figure~\ref{fig:posterior_noise_sweep} shows two complementary aspects of robustness under observational noise. The Hellinger distance quantifies the change in the posterior induced by the addition of observational noise, whereas the posterior mean error measures inference accuracy relative to the true source location. Because these metrics characterize different aspects of posterior behavior, they should be interpreted jointly. A small Hellinger distance therefore indicates that the posterior is relatively insensitive to the added noise, but does not by itself imply accurate inference.

The Symmetrized Sinkhorn posterior exhibits the most favorable behavior overall. Its Hellinger distance remains consistently smaller than that of the competing methods, while the corresponding interquartile band remains narrow across the noise sweep, indicating stable behavior across stochastic realizations of the data. At the same time, it achieves the smallest posterior mean error throughout the experiment, again with limited variability across seeds. Thus, the reduced sensitivity to spatial and temporal misalignment provided by joint space-time transport yields a posterior that remains both stable and accurately localized in the presence of noise.

The Trace Wasserstein posterior occupies an intermediate regime. Its posterior mean error is substantially smaller than that of the Euclidean posterior, indicating that trace-wise transport successfully mitigates part of the cycle-skipping behavior present in the Euclidean risk landscape. However, both the Hellinger distance and the posterior mean error display noticeably larger variability across seeds, especially as the noise level increases. This behavior is consistent with the remaining nonconvexity of the trace-wise Wasserstein risk landscape: temporal misalignment is reduced, but competing spatial alignments can still induce instability in the posterior structure. In contrast, the Euclidean posterior exhibits comparatively moderate Hellinger distances because the addition of observational noise induces relatively little further change to an already inaccurate posterior. This is reflected in the consistently large posterior mean error across the entire noise sweep. Combined with the noise-free resolution study, these results indicate that the Euclidean posterior remains stably concentrated around incorrect modes induced by the highly nonconvex Euclidean risk landscape.

\subsection{Multiple-event inference: population-level recovery}
\label{sec:population_inference}

We now consider the multi-event setting, in which each dataset $g_j$ is generated by a parameter $\theta_j$ drawn from an underlying distribution $\mu_\theta$. The objective is no longer to recover a single parameter, but rather to infer properties of the distribution $\mu_\theta$ from a finite collection of observations. 
This setting introduces a fundamentally different validation regime. Unlike the single-event case, where performance is assessed through posterior concentration around a fixed parameter, here we can directly compare inferred distributions with the true underlying law. This enables a quantitative assessment of inference quality at finite sample sizes.

We consider a population distribution $\mu_\theta$, for the parameter $\theta$, given by a truncated Gaussian law on $[0.2,0.8]$, centered at $\theta_c = 0.35$ with standard deviation $\tau = 0.1$. For each experiment, we generate
\[
\theta_j \overset{\mathrm{iid}}{\sim} \mu_\theta,
\qquad
j=1,\dots,J,
\]
together with the corresponding datasets $\{ g_j \}_{j=1}^J$ on the coarse observation grid ($21\times21$). A relative noise level of $5\%$ is added to each dataset. For every event, we compute the posterior $\Pi_j(\theta \mid g_j)$ using each of the three loss functions.

From these posteriors, we construct an estimator of the population distribution (i.e., the posterior-mixture measure):
\[
\mu_{\theta}^J =
\frac{1}{J}\sum_{j=1}^J \Pi_j(\theta \mid g_j),
\]
and compare it against the true distribution $\mu_\theta$. As a reference, we also compare the true distribution with the oracle empirical distribution
\[
\mu_{\theta}^{J,\mathrm{oracle}}=
\frac{1}{J}\sum_{j=1}^J \delta_{\theta_j},
\]
which provides a natural finite-sample reference based on the true event parameters. All distributional errors are measured using the quadratic Wasserstein distance.

Figure~\ref{fig:population_inference} shows the Wasserstein error as a function of the number of events $J$, summarized over 50 independent repetitions. Solid curves indicate the median across repetitions, shaded regions correspond to the interquartile range (25\%--75\% quantiles), and the thick oracle curve represents the intrinsic finite-sample approximation error.

\begin{figure}[!ht]
\centering
\includegraphics[width=0.68\textwidth]{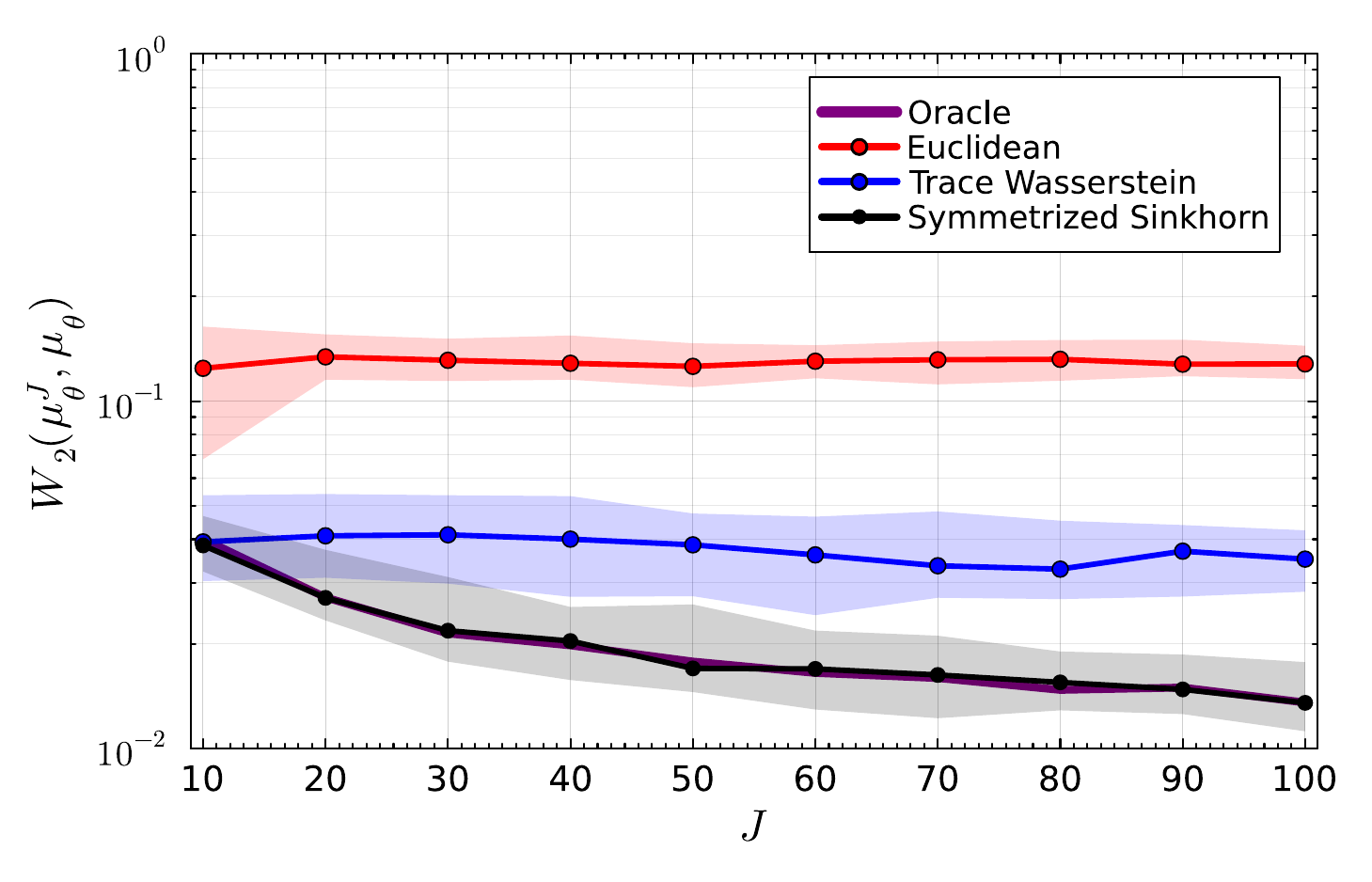}
\caption{Population-level recovery of the posterior-mixture population estimator as a function of the number of observed events $J$. The vertical axis shows the quadratic Wasserstein distance to the true population distribution $\mu_\theta$. Thin solid curves show the median across 50 independent repetitions, shaded regions indicate the interquartile range (25\%--75\% quantiles), and the thick oracle curve represents the empirical distribution built from the true event parameters. The oracle curve is nearly indistinguishable from the median Symmetrized Sinkhorn curve because the two almost completely overlap.}
\label{fig:population_inference}
\end{figure}

As expected, the oracle empirical distribution improves as the number of events increases, reflecting the intrinsic sampling variability associated with finite observations. The oracle curve is not a strict lower bound for all possible estimators, but it provides a useful reference scale based on the unobserved true event parameters.

The Euclidean posterior exhibits the weakest population-level performance. Its Wasserstein error remains large and nearly flat as $J$ increases, indicating that adding more events does not remove the bias introduced by inaccurate event-level posteriors.

The Trace Wasserstein posterior substantially improves population recovery relative to the Euclidean posterior. Its error is smaller across the full range of $J$, but it remains separated from the oracle reference and displays a wider interquartile band. This suggests that trace-wise transport reduces part of the event-level misalignment, while residual ambiguity still affects the recovered population law.

In contrast, the Symmetrized Sinkhorn posterior yields the most accurate and stable population-level inference. Its posterior-mixture error closely follows the oracle reference curve across the full range of $J$, with comparatively narrow interquartile bands. This indicates that the improved event-level risk landscapes induced by joint transport in space and time not only enhances single-event inference, but also enables reliable recovery of population-level structure from finite noisy data.

\subsection{Multi-parameter extension}
\label{sec:multi_parameter}

We finally consider a multi-parameter inverse problem that illustrates the adaptive Gibbs sampling framework introduced in Section~\ref{sec:adaptive_mcmc}. In contrast to the previous experiments, the inverse temperature is treated as an unknown parameter and inferred jointly with the physical parameters. 
In addition to the source location, we allow the dominant waveform frequency and secondary arrival structure to vary. The unknown parameter vector is
\[
\boldsymbol{\theta} = (x_s,\nu,\Delta_2),
\]
where \(x_s\) denotes the source location, \(\nu\) controls the dominant frequency of the emitted waveform, and \(\Delta_2\) introduces variability in the timing of secondary arrivals. This setting leads to a higher-dimensional inverse problem in which travel-time effects, oscillatory waveform structure, and secondary arrivals interact nonlinearly.

To avoid repeating the one-dimensional studies, we focus here on a representative noisy and moderately resolved setting. In all experiments in this section, the observed data are corrupted with \(5\%\) noise and recorded on a \(41\times41\) grid. The purpose of these experiments is not to re-examine the effects of noise or discretization resolution, but to investigate whether the favorable posterior behavior observed for the Sinkhorn loss persists in the presence of parameter coupling.

As discussed in Section~\ref{sec:adaptive_mcmc}, the Gibbsian framework accommodates a broad class of empirical risks. Here we compare posteriors associated with the Euclidean, trace-wise Wasserstein, and symmetrized Sinkhorn losses. For each loss, we consider the Gibbsian posterior density
\[
\pi(\boldsymbol{\theta},\lambda \mid y)
\propto
\exp \left[-\lambda n \Phi(\boldsymbol{\theta};y)\right]
\pi_0(\boldsymbol{\theta})\pi_0(\lambda),
\]
where \(\lambda>0\) is an inverse-temperature parameter. Rather than fixing \(\lambda\) manually, we infer it jointly with the physical parameters. The sampler alternates between a multivariate Metropolis update for \(\boldsymbol{\theta}\) with \(\lambda\) fixed, and a one-dimensional Metropolis update for \(\lambda\) with \(\boldsymbol{\theta}\) fixed. The proposal covariance for \(\boldsymbol{\theta}\) is adapted from the empirical covariance of the chain, while the proposal variance for \(\log\lambda\) is adapted separately as a scalar quantity. This adaptive construction reduces sensitivity to manual temperature tuning and is particularly important in the multi-parameter setting, where different parameter directions may exhibit substantially different effective curvatures.

The Euclidean posterior exhibited severe sensitivity to parameter coupling and consistently collapsed toward incorrect highly concentrated modes. Since this behavior was qualitatively different from the transport-based posteriors and visually dominated the scales of the marginal plots, we omit the detailed \(L^2\) posterior figures and focus instead on the comparison between the trace-wise Wasserstein and Sinkhorn Gibbs posteriors.

Figures~\ref{fig:multiparameter_xs8}--\ref{fig:multiparameter_xs4} display marginal posterior densities for two representative parameter configurations. In the first case, shown in Figure~\ref{fig:multiparameter_xs8}, the trace-wise Wasserstein posterior successfully localizes the source position \(x_s\), but remains substantially broader in the waveform-frequency and secondary-arrival parameters. In contrast, the Sinkhorn posterior remains sharply concentrated around the true parameters across all marginals. The second case, shown in Figure~\ref{fig:multiparameter_xs4}, is more challenging. Here the trace-wise Wasserstein posterior becomes visibly biased away from the true parameter vector, particularly in \(x_s\) and \(\Delta_2\), while also exhibiting substantially larger posterior uncertainty. In contrast, the Sinkhorn posterior remains accurately centered and tightly concentrated near the true parameter values across all three marginals.

These results suggest that the advantages of the Sinkhorn-based risk are not limited to the scalar source-location problem. Even when multiple waveform and arrival-time parameters are inferred simultaneously, the Sinkhorn-based Gibbs posterior remains accurately centered around the true parameter values while exhibiting less bias and sharper posterior concentration under noisy and moderately resolved observations.
\begin{figure}[t]
\centering
\includegraphics[width=0.32\textwidth]{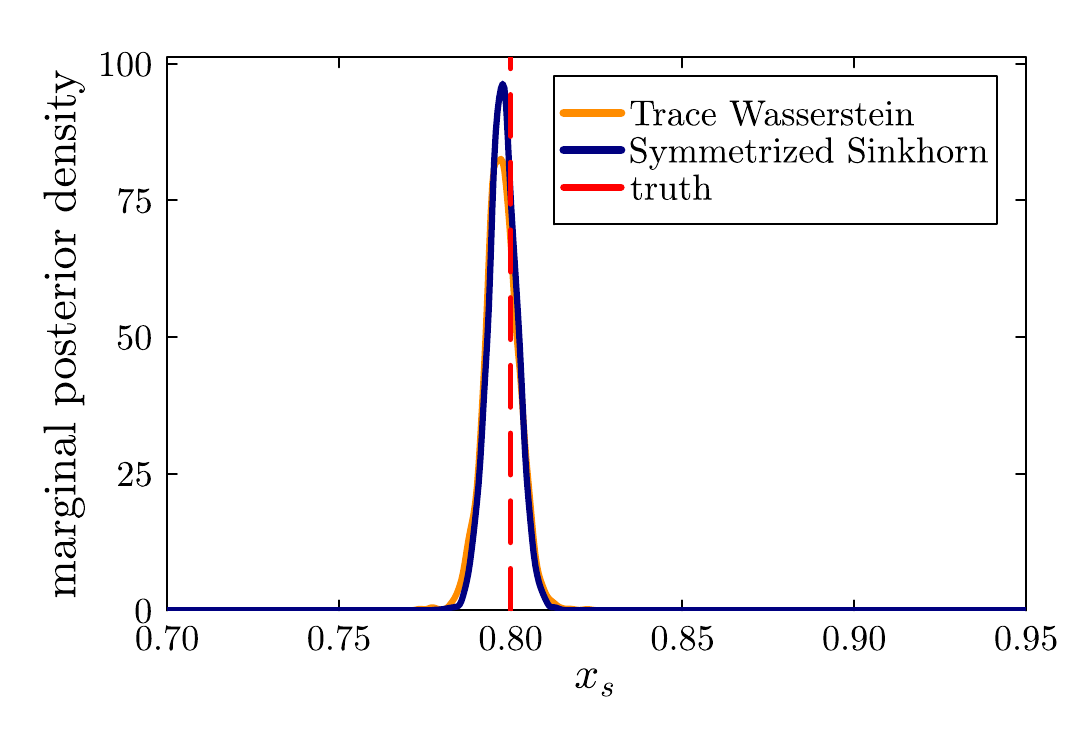}
\hfill
\includegraphics[width=0.32\textwidth]{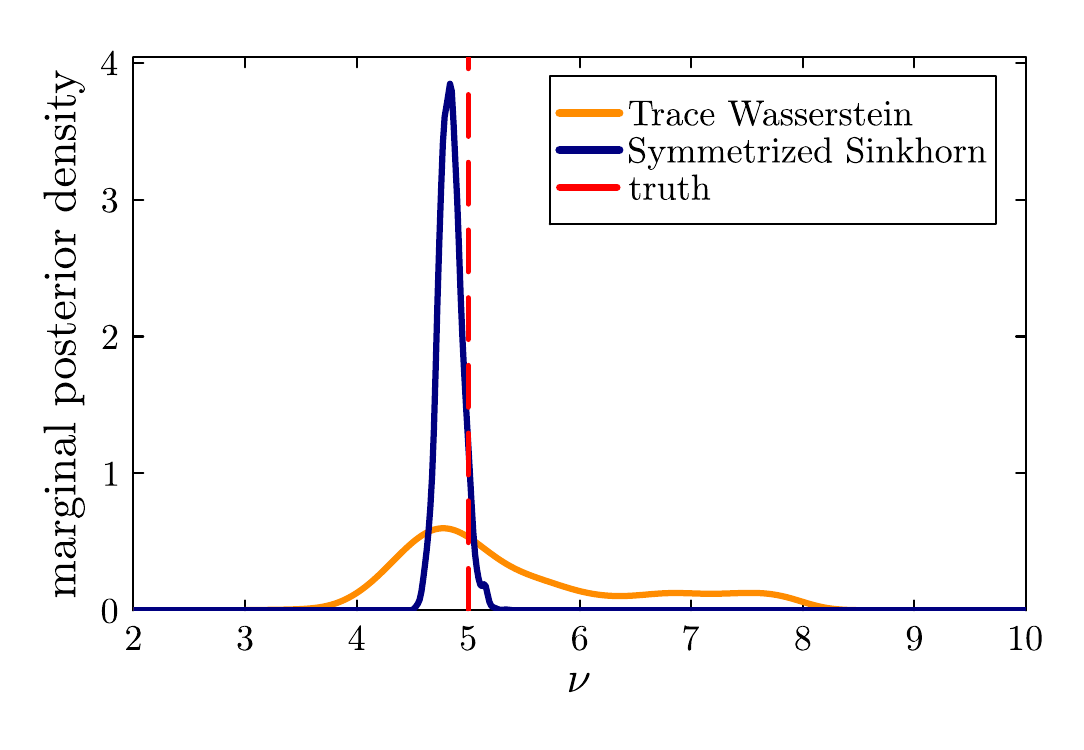}
\hfill
\includegraphics[width=0.32\textwidth]{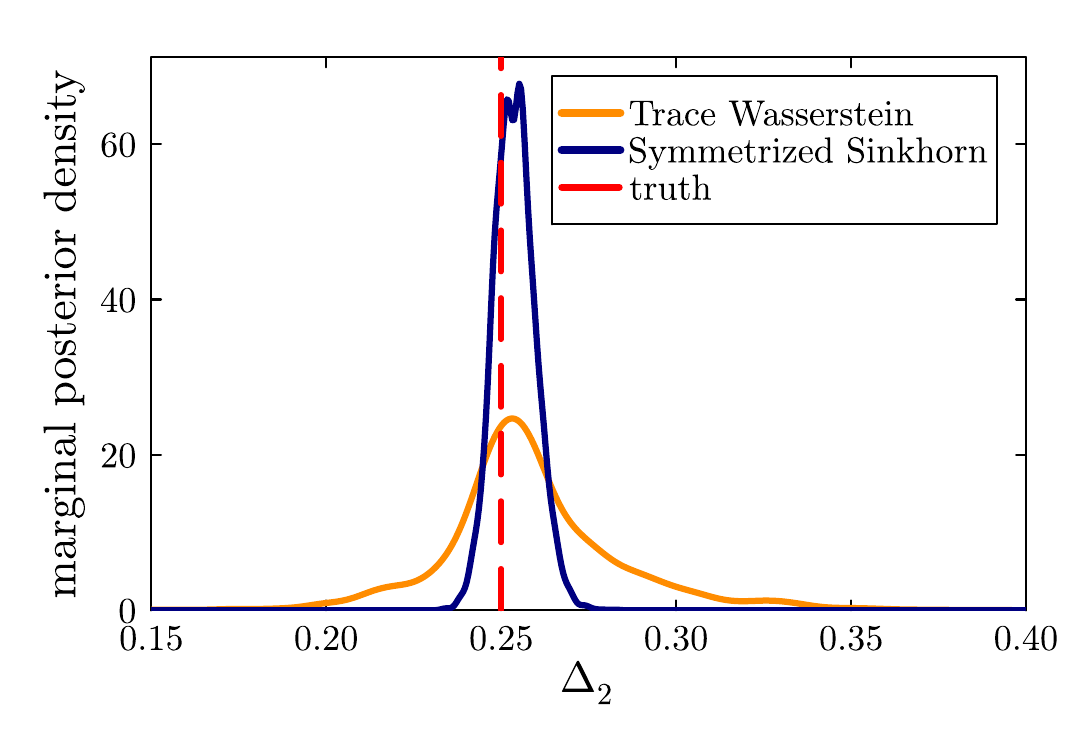}
\caption{Marginal posterior densities for the multi-parameter event-level problem with true parameter vector $\boldsymbol{\theta}_{\rm true}=(0.8,5.0,0.25)$ on a $41\times41$ grid with $5\%$ noise. The trace-wise Wasserstein posterior accurately localizes the source position but remains broader and less accurate in the waveform-frequency and secondary-arrival parameters. The Sinkhorn posterior remains substantially more concentrated around the true parameter values across all marginals.}
\label{fig:multiparameter_xs8}
\end{figure}

\begin{figure}[t]
\centering
\includegraphics[width=0.32\textwidth]{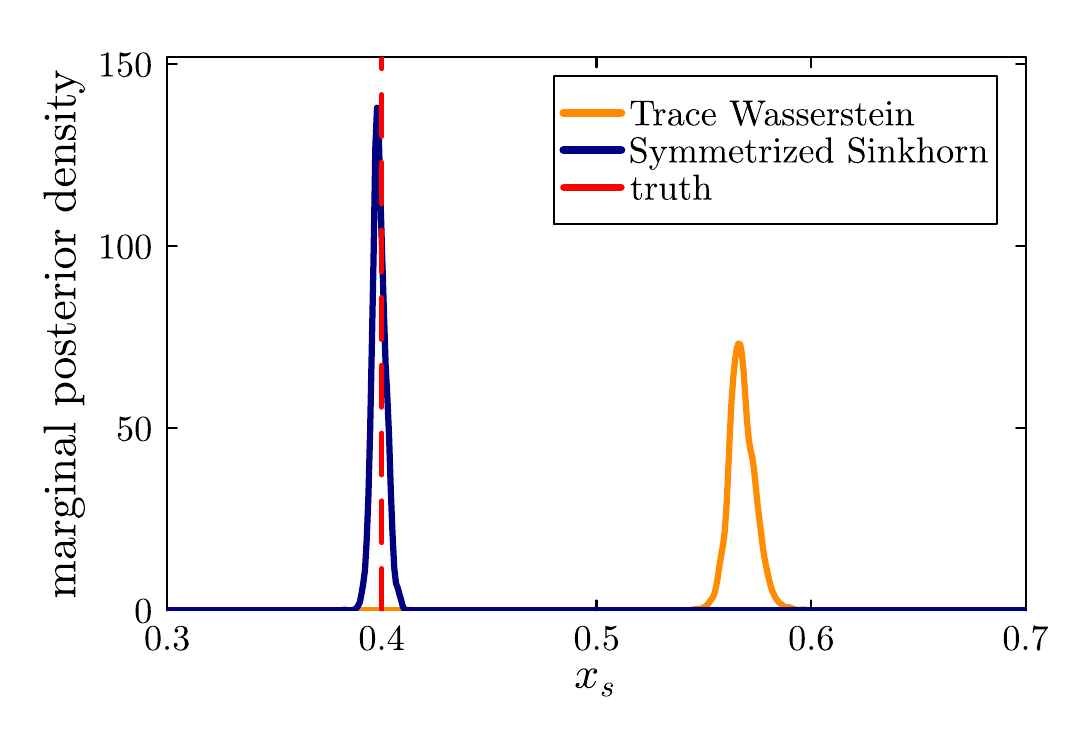}
\hfill
\includegraphics[width=0.32\textwidth]{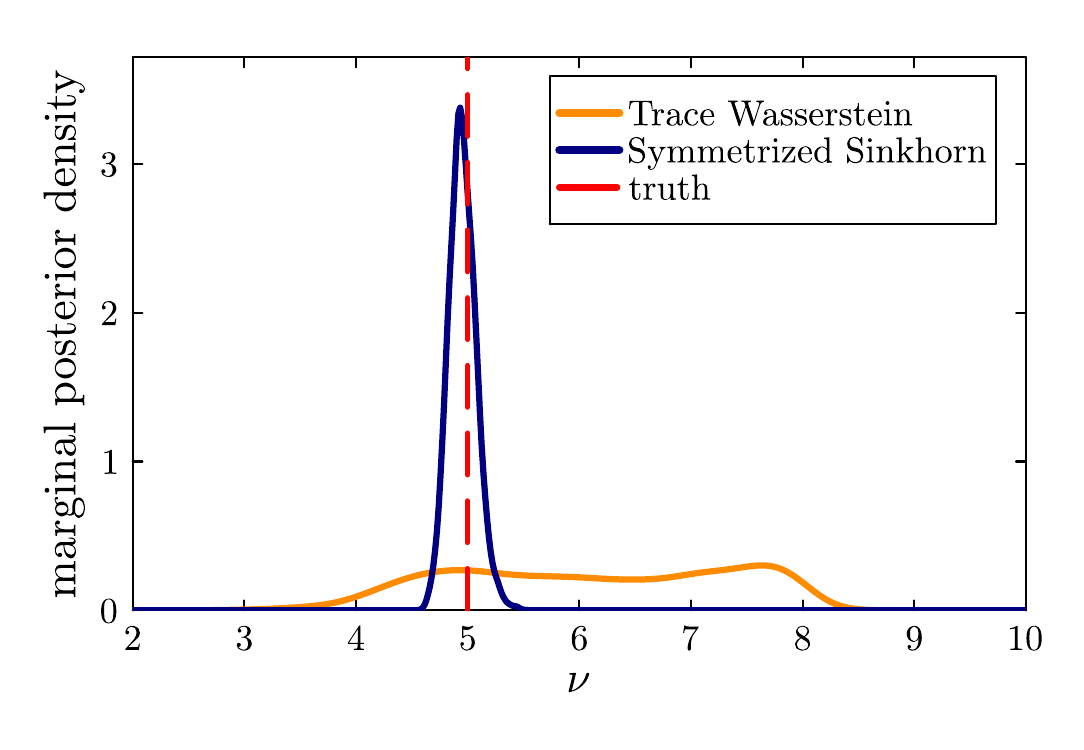}
\hfill
\includegraphics[width=0.32\textwidth]{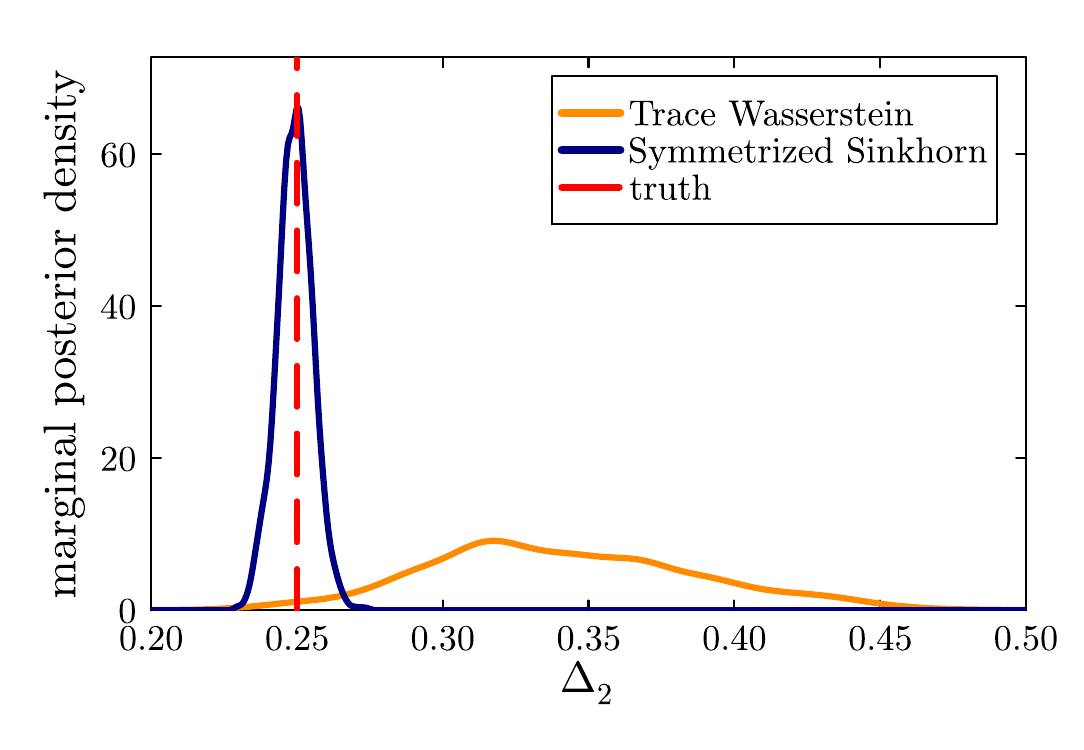}
\caption{Marginal posterior densities for the multi-parameter event-level problem with true parameter vector $\boldsymbol{\theta}_{\rm true}=(0.4,5.0,0.25)$ on a $41\times41$ grid with $5\%$ noise. In this more challenging configuration, the trace-wise Wasserstein posterior becomes biased away from the true parameter vector, particularly in $x_s$ and $\Delta_2$, while the Sinkhorn posterior remains accurately centered and sharply concentrated near the truth across all three parameters.}
\label{fig:multiparameter_xs4}
\end{figure}

\section{Conclusions}
\label{sec:conclusions}

We have introduced a symmetrized Sinkhorn-Gibbs inference framework for oscillatory inverse problems. The proposed approach integrates Gibbsian inference with entropic optimal transport through a normalization procedure for signed oscillatory signals and a symmetrized Sinkhorn loss that incorporates transport information from both positive and negative signal excursions. 
This construction yields a Gibbs posterior specifically designed to address optimization challenges arising from oscillatory data, including nonconvexity, spurious local minima, and cycle-skipping effects that commonly degrade conventional discrepancy measures.

On the theoretical side, we established smoothness properties of the normalized Sinkhorn divergence, proved well-definedness of the resulting Gibbs posterior, and derived robustness guarantees with respect to perturbations in both the observed data and the forward model. On the computational side, numerical experiments demonstrated that the proposed symmetrized Sinkhorn-Gibbs framework produces less oscillatory empirical risk landscapes with fewer spurious local minima, leading to posterior distributions that are more accurately centered around the true parameters, while also providing greater robustness to observational noise and improved population-level recovery than Gibbs inference based on Euclidean and trace-wise Wasserstein losses.

An important direction for future research is the extension of the proposed framework to high-dimensional parameter spaces. Developing scalable variational inference methods for symmetrized Sinkhorn-Gibbs posteriors represents a promising avenue for enabling transport-based inference to scale to large uncertainty quantification problems while preserving its ability to account for spatial and temporal displacement in the feature domain.

\section*{Acknowledgements}
M. Motamed was supported by the U.S. Department of Energy (DOE), Office of Science, Advanced Scientific Computing Research, under Award No. DE-SC0025481, and by the National Science Foundation under Grant No. DMS-2436318. 

This paper describes objective technical results and analysis. Any subjective views or opinions that might be expressed in the paper do not necessarily represent the views of the U.S. Department of Energy or the United States Government.

This article has been authored by an employee of National Technology \& Engineering Solutions of Sandia, LLC under Contract No. DE-NA0003525 with the DOE. The employee owns all right, title and interest in and to the article and is solely responsible for its contents. The United States Government retains and the publisher, by accepting the article for publication, acknowledges that the United States Government retains a non-exclusive, paid-up, irrevocable, world-wide license to publish or reproduce the published form of this article or allow others to do so, for United States Government purposes. The DOE will provide public access to these results of federally sponsored research in accordance with the DOE Public Access Plan https://www.energy.gov/downloads/doe-public-access-plan

\section*{Data Availability Statement}
The data and code supporting the findings of this study are available from the corresponding author upon reasonable request and will be made publicly available upon acceptance of the manuscript.

\section*{Author contributions}

M. Motamed and G. Huerta jointly conceived the project and contributed to the development of the underlying ideas. The methodology, theoretical analysis, computational investigations, and initial manuscript preparation were carried out by M. Motamed. G. Huerta contributed to methodological discussions, interpretation of the results, and manuscript review and editing. Both authors contributed to the revision of the manuscript and approved the final version.

\section*{Conflict of Interest}

The authors declare no conflicts of interest.

\section*{Ethics Statement}

Ethical approval was not required for this study.

\appendix

\section{Hellinger Stability Estimate for Symmetrized Sinkhorn-Gibbs Posteriors}
\label{app:hellinger_stability}

In this appendix, we provide the complete proof of Theorem~\ref{stability_data}, which establishes robustness of the symmetrized Sinkhorn-Gibbs posterior with respect to perturbations in the observed data. The proof follows the general strategy used in the stability analysis of Bayesian inverse problems \cite{Stuart:10,Sullivan:2017}, adapted to the Gibbsian setting and the symmetrized Sinkhorn risk.

\medskip

\noindent
{\bf Complete proof of Theorem~\ref{stability_data}.} Denote the risk functionals
\[
\Phi(\boldsymbol{\theta}; g)
=
\mathcal D_{\sigma,\varepsilon}^{2}
\bigl(
f(\boldsymbol{\theta}),g
\bigr),
\qquad
\Phi(\boldsymbol{\theta}; g')
=
\mathcal D_{\sigma,\varepsilon}^{2}
\bigl(
f(\boldsymbol{\theta}),g'
\bigr),
\]
with corresponding normalizing constants
\[
Z = \int_{\Theta} \exp\big(-\lambda \, \Phi(\boldsymbol{\theta}; g)\big) \, \Pi_0(d\boldsymbol{\theta}), \quad Z' = \int_{\Theta} \exp\big(-\lambda \, \Phi(\boldsymbol{\theta}; g')\big) \, \Pi_0(d\boldsymbol{\theta}),
\]
and define the posterior densities with respect to the prior measure $\Pi_0$ by
\[
p(\boldsymbol{\theta})
=
\frac{
\exp(-\lambda\,\Phi(\boldsymbol{\theta};g))
}{Z},
\qquad
p'(\boldsymbol{\theta})
=
\frac{
\exp(-\lambda\,\Phi(\boldsymbol{\theta};g'))
}{Z'}.
\]
The squared Hellinger distance between $\Pi$ and $\Pi'$ is
\[
d_{\mathrm H}(\Pi,\Pi')^2
=
\int_\Theta
\left(
\sqrt{p(\boldsymbol{\theta})}
-
\sqrt{p'(\boldsymbol{\theta})}
\right)^2
\,\Pi_0(d\boldsymbol{\theta}).
\]
To estimate this quantity, we write
\[
\sqrt{p(\boldsymbol{\theta})}
-
\sqrt{p'(\boldsymbol{\theta})}
=
\frac{
\sqrt{\exp(-\lambda\,\Phi(\boldsymbol{\theta};g))}
}{\sqrt Z}
-
\frac{
\sqrt{\exp(-\lambda\,\Phi(\boldsymbol{\theta};g'))}
}{\sqrt{Z'}}.
\]
Adding and subtracting $\sqrt{\exp(-\lambda \Phi(\boldsymbol{\theta}; g))}/\sqrt{Z'}$ yields
\[
\sqrt p-\sqrt{p'}
=
\frac{
\sqrt{\exp(-\lambda\,\Phi(g))}
-
\sqrt{\exp(-\lambda\,\Phi(g'))}
}{\sqrt{Z'}}
+
\sqrt{\exp(-\lambda\,\Phi(g))}
\left(
\frac1{\sqrt Z}
-
\frac1{\sqrt{Z'}}
\right),
\]
where we abbreviate $\Phi(\boldsymbol{\theta};g)$ by $\Phi(g)$ and similarly for $\Phi(g')$.

We estimate the two contributions separately. 
First, since the function $x \mapsto \exp(-\lambda x/2)$ is globally Lipschitz on bounded intervals, and since $\Phi$ is uniformly bounded by Theorem~\ref{well-defined}, there exists a constant $C_\lambda>0$ such that
\[
\left|
\sqrt{\exp(-\lambda\,\Phi(g))}
-
\sqrt{\exp(-\lambda\,\Phi(g'))}
\right|
\le
C_\lambda
\left|
\Phi(g)-\Phi(g')
\right|.
\]
Moreover, by Corollary~\ref{cor:sym_risk_stability},
\[
\left|
\Phi(\boldsymbol{\theta};g)
-
\Phi(\boldsymbol{\theta};g')
\right|
=
\left|
\mathcal D_{\sigma,\varepsilon}^{2}
\bigl(
f(\boldsymbol{\theta}),g
\bigr)
-
\mathcal D_{\sigma,\varepsilon}^{2}
\bigl(
f(\boldsymbol{\theta}),g'
\bigr)
\right|
\le
C_\varepsilon
\,
\|g-g'\|_{\dot H^{-1}(X)}
\]
uniformly for all $\boldsymbol{\theta}\in\Theta$. 
Therefore,
\[
\left|
\sqrt{\exp(-\lambda\,\Phi(g))}
-
\sqrt{\exp(-\lambda\,\Phi(g'))}
\right|
\le
C_\lambda C_\varepsilon
\,
\|g-g'\|_{\dot H^{-1}(X)}.
\]
Squaring and integrating yields
\[
\int_\Theta
\left(
\frac{
\sqrt{\exp(-\lambda\,\Phi(g))}
-
\sqrt{\exp(-\lambda\,\Phi(g'))}
}{\sqrt{Z'}}
\right)^2
\Pi_0(d\boldsymbol{\theta})
\le
\frac{
C_\lambda^2 C_\varepsilon^2
}{Z'}
\,
\|g-g'\|_{\dot H^{-1}(X)}^2.
\]

We next estimate the contribution arising from the normalization constants. Using
\[
\left|
\frac1{\sqrt Z}
-
\frac1{\sqrt{Z'}}
\right|
\le
\frac{
|Z-Z'|
}{
2\,\min(Z,Z')^{3/2}
},
\]
it suffices to bound $|Z-Z'|$. 
Using the mean value theorem,
\begin{align*}
|Z-Z'|
&=
\left|
\int_\Theta
\Bigl(
\exp(-\lambda\,\Phi(g))
-
\exp(-\lambda\,\Phi(g'))
\Bigr)
\,\Pi_0(d\boldsymbol{\theta})
\right| \\
&\le
\lambda
\int_\Theta
\left|
\Phi(g)-\Phi(g')
\right|
\max\!\Bigl(
\exp(-\lambda\,\Phi(g)),
\exp(-\lambda\,\Phi(g'))
\Bigr)
\Pi_0(d\boldsymbol{\theta}).
\end{align*}
Since the exponential factors are bounded above by one and
\[
\left|
\Phi(g)-\Phi(g')
\right|
\le
C_\varepsilon
\,
\|g-g'\|_{\dot H^{-1}(X)},
\]
we obtain
\[
|Z-Z'|
\le
\lambda C_\varepsilon
\,
\|g-g'\|_{\dot H^{-1}(X)}.
\]

Furthermore, Theorem~\ref{well-defined} implies that $\Phi$ is uniformly bounded on $\Theta$. Hence there exists $M<\infty$ such that
\[
0 \le \Phi(\boldsymbol{\theta};g),
\Phi(\boldsymbol{\theta};g')
\le M
\qquad
\forall \boldsymbol{\theta}\in\Theta.
\]
Consequently,
\[
Z,Z'
\ge
e^{-\lambda M}
\Pi_0(\Theta)
=
e^{-\lambda M},
\]
so both normalization constants are bounded below by a strictly positive constant independent of the perturbation size. 
Therefore,
\[
\left|
\frac1{\sqrt Z}
-
\frac1{\sqrt{Z'}}
\right|
\le
C
\,
\|g-g'\|_{\dot H^{-1}(X)}
\]
for some constant $C>0$.

Since $\sqrt{\exp(-\lambda\,\Phi(g))}
\le 1$, the second contribution to the Hellinger distance satisfies
\[
\int_\Theta
\left(
\sqrt{\exp(-\lambda\,\Phi(g))}
\left(
\frac1{\sqrt Z}
-
\frac1{\sqrt{Z'}}
\right)
\right)^2
\Pi_0(d\boldsymbol{\theta})
\le
C
\,
\|g-g'\|_{\dot H^{-1}(X)}^2.
\]

Combining the two estimates and absorbing all constants into a generic constant $C>0$ yields
\[
d_{\mathrm H}(\Pi,\Pi')^2
\le
C
\,
\|g-g'\|_{\dot H^{-1}(X)}^2.
\]
Taking square roots gives
\[
d_{\mathrm H}(\Pi,\Pi')
\le
C
\,
\|g-g'\|_{\dot H^{-1}(X)},
\]
which establishes the claimed robustness estimate.

\bibliographystyle{unsrt}
\bibliography{refs_motamed}

\end{document}